\documentclass[11pt]{article}

\usepackage[a4paper,margin=1in]{geometry}
\usepackage[T1]{fontenc}
\usepackage[utf8]{inputenc}
\usepackage{amsmath,amssymb,amsthm,mathtools}
\usepackage{enumitem}
\usepackage[colorlinks=true,linkcolor=blue,citecolor=blue,urlcolor=blue]{hyperref}

\newcommand{\R}{\mathbb{R}}
\DeclareMathOperator{\supp}{supp}
\newcommand{\norm}[1]{\lVert #1\rVert}
\newcommand{\abs}[1]{\lvert #1\rvert}

\theoremstyle{plain}
\newtheorem{theorem}{Theorem}[section]
\newtheorem{proposition}[theorem]{Proposition}
\newtheorem{lemma}[theorem]{Lemma}

\theoremstyle{definition}

\theoremstyle{remark}

\title{Characteristic localization, sharp lifespan asymptotics and global dynamics
for a one-dimensional derivative wave equation}

\author{
Firas Kaabi\\
\small Department of Mathematics, Faculty of Sciences of Tunis,\\
\small University of Tunis El Manar, LR Analyse Non-Lin\'eaire et G\'eom\'etrie, LR21ES08,\\
\small El Manar 2, 2092 Tunis, Tunisia\\
\small \texttt{firaskaabi17@gmail.com}
}

\date{}

\begin{document}

\maketitle

\begin{abstract}
We study the Cauchy problem
$$
v_{tt}-v_{xx}=\abs{v_t+v_x}^{m}\abs{v_t}^{n},
\qquad
v(x,0)=\eta\varphi(x),\quad v_t(x,0)=\eta\psi(x),
$$
on $\R$, with compactly supported profiles $\varphi\in C_0^2(\R)$,
$\psi\in C_0^1(\R)$, exponents $m,n>1$, and amplitude $\eta>0$. We show that the sign of
$P_0=\psi+\varphi'$ decides the behaviour of small solutions, whatever the sign of
$\psi-\varphi'$. If $P_0$ is positive at some point, the lifespan $T(\eta)$ obeys
explicit two-sided bounds of order $\eta^{-(m+n-1)}$ for every $\eta>0$, and
$$
\lim_{\eta\to0^+}\eta^{m+n-1}T(\eta)=\frac{2^{n}}{(m+n-1)\,(\max_{\R}P_0)^{m+n-1}} .
$$
If $P_0\le0$, the solution is global for every amplitude below an explicit threshold. If
moreover $P_0\not\equiv0$, the component $v_t+v_x$ decays at the universal rate
$\bigl(2^{n}/((m+n-1)t)\bigr)^{1/(m+n-1)}$, and the gradient of the solution converges
uniformly to that of a free wave travelling to the right; if $P_0\equiv0$, the solution is
itself such a travelling wave. The analysis rests on a localization property: the zero
set of $v_t+v_x$ is invariant along its own characteristics, so that the nonlinear source
stays in a slab of fixed width moving with speed one, and each characteristic of the other
family is forced only during a bounded time. This yields the global existence, the
long-time behaviour and the exact value of the limit. The results extend to the endpoint
exponents $m,n\ge1$ and to sources $f(v_t+v_x)g(v_t)$, for which $T(\eta)$ is asymptotic
to an Osgood-type integral.
\end{abstract}

\noindent\textbf{Keywords:} nonlinear wave equation; derivative nonlinearity;
characteristic localization; lifespan; finite-time blow-up; global existence;
asymptotic behaviour; Osgood criterion; one space dimension.

\noindent\textbf{Mathematics Subject Classification (2020):}
Primary 35L71, 35B44; Secondary 35B30, 35B40, 35L05.

\section{Introduction}\label{sec:intro}

In one space dimension the wave operator factors as
$\partial_t^2-\partial_x^2=(\partial_t-\partial_x)(\partial_t+\partial_x)$, and the two
operators $\partial_t\pm\partial_x$ transport the two travelling components of a free
wave. A nonlinear source built from first derivatives may involve those components in
quite different ways, and the way in which it does so determines the lifespan of small
solutions. We are concerned here with the source $\abs{v_t+v_x}^{m}\abs{v_t}^{n}$, one
factor of which is a derivative along a single characteristic direction, while the other
is the velocity and mixes both directions.

The Cauchy problem under consideration is
\begin{equation}\label{eq:main}
\begin{cases}
v_{tt}-v_{xx}=\abs{v_t+v_x}^{m}\abs{v_t}^{n},
& x\in\R,\ t>0,\\[2pt]
v(x,0)=\eta\varphi(x),
& x\in\R,\\[2pt]
v_t(x,0)=\eta\psi(x),
& x\in\R,
\end{cases}
\end{equation}
with the standing assumptions
\begin{equation}\label{eq:standing}
m>1,\qquad n>1,\qquad
\varphi\in C_0^2(\R),\quad \psi\in C_0^1(\R),\quad
(\varphi,\psi)\not\equiv(0,0),
\end{equation}
and
\begin{equation}\label{eq:support}
\supp\varphi\cup\supp\psi\subset[-\rho,\rho],
\qquad \rho>0.
\end{equation}
The exponents and the profiles are fixed once and for all, and the parameter $\eta>0$
measures the amplitude of the data. A classical solution on $\R\times[0,T]$ is a function
$v\in C^2(\R\times[0,T])$ satisfying \eqref{eq:main} at every point, and the lifespan is
\begin{equation}\label{eq:lifespan-def}
T(\eta):=\sup\bigl\{T>0:\ \text{a classical solution of \eqref{eq:main}
exists on }\R\times[0,T]\bigr\}.
\end{equation}
Section~\ref{sec:extensions} relaxes \eqref{eq:standing} to $m,n\ge1$ and replaces the
powers by general nonlinear factors.

The two factors of the source play different roles. If
$v(x,t)=A_+(x+t)+A_-(x-t)$ is a free wave, then
$$
v_t+v_x=2A_+'(x+t),
\qquad
v_t-v_x=-2A_-'(x-t),
$$
so that $v_t+v_x$ measures one travelling component alone, whereas
$v_t=\frac12\bigl[(v_t+v_x)+(v_t-v_x)\bigr]$ measures both. The question addressed below
is whether the single-component factor is strong enough to govern the small-data
behaviour of \eqref{eq:main}, and in particular to produce a sharp lifespan.

For one-dimensional equations $u_{tt}-u_{xx}=F(Du,D_xDu)$ with a smooth nonlinearity
vanishing to order $p$ at the origin, and for data of size $\varepsilon$, the general
theory of Li, Yu and Zhou \cite{LiYuZhou1991,LiYuZhou1992} yields the lower bound
$T(\varepsilon)\ge c\varepsilon^{-(p-1)}$. Several model equations show that this scale is
optimal. For $u_{tt}-u_{xx}=\abs{u_t}^{p}$ with $p>1$, the matching upper bound is due to
Zhou \cite{Zhou2001} and rests on a pointwise comparison with an ordinary differential
equation, while the lower bound for every $p>1$, including the values of $p$ for which
the nonlinear term is not smooth, is contained in the work of Kitamura, Morisawa and
Takamura \cite{KitamuraMorisawaTakamura2023}. For the spatial-derivative source
$u_{tt}-u_{xx}=\abs{u_x}^{q}$ with $q>1$, the same order was proved by Sasaki, Takamatsu
and Takamura \cite{SasakiTakamatsuTakamura2023}; there the pointwise positivity used in
the previous case is unavailable, and the blow-up part relies instead on a weighted
functional of the solution, of the type introduced by Rammaha
\cite{Rammaha1995,Rammaha1997}. Non-autonomous versions of this problem were treated
recently in \cite{LaiRenSasakiTakamura2026}.

For the product source $u_{tt}-u_{xx}=\abs{u_t}^{p}\abs{u_x}^{q}$ with $p,q>1$, Haruyama,
Sasaki and Takamura \cite{HaruyamaSasakiTakamura2025} established the lower bound
$T(\varepsilon)\ge c\varepsilon^{-(p+q-1)}$ and observed that a matching upper bound is
not available, the two mechanisms just described being incompatible with one another. The
situation is different when the source depends on a single characteristic component. For
$u_{tt}-u_{xx}=\abs{u_t\pm u_x}^{p-1}(u_t\pm u_x)$ with data
$u(x,0)=\varepsilon a(x)$, $u_t(x,0)=\varepsilon b(x)$, the equation reduces along the
corresponding characteristic family to a scalar ordinary differential equation, so that
$T(\varepsilon)\le C\varepsilon^{-(p-1)}$ whenever $\pm a'(x_0)+b(x_0)\neq0$ for some
$x_0$, whereas the travelling wave $\varepsilon a(x\mp t)$ is a global solution when
$\pm a'+b\equiv0$; see again \cite{HaruyamaSasakiTakamura2025}. Further one-dimensional
lifespan results for weighted and combined nonlinearities are contained in
\cite{MorisawaSasakiTakamura2023,KidoSasakiTakamatsuTakamura2024,Kitamura2026,
HaruyamaTakamura2025}, and the survey \cite{Takamura2023} describes the state of the
theory. For blow-up mechanisms in higher dimensions we refer to
\cite{John1979,Sideris1984,Rammaha1987,John1990,Alinhac1995,Hormander1997,Takamura2015}.

Problem \eqref{eq:main} lies between the last two situations. Its second factor produces
the coupling of the two characteristic components that obstructs the product problem,
while its first factor is attached to a single family. We show that, with the velocity as
second factor, the single-component factor decides the small-data behaviour completely.
A first version of this work, which contains the two-sided lifespan bound under the
one-sided condition $\psi-\varphi'\ge0$, was circulated as \cite{Kaabi2026}; the present
paper removes that condition and adds the exact asymptotics of the lifespan, the global
existence and the long-time behaviour described below.

Our results are expressed through the characteristic components $\psi+\varphi'$ and
$\psi-\varphi'$ of the data, and the precise statement is given in Section~\ref{sec:results}. We prove that the sign of $\psi+\varphi'$ alone decides the
behaviour of solutions of small amplitude. If $\psi+\varphi'$ is positive at some point,
the solution blows up, the lifespan satisfies explicit two-sided bounds of order
$\eta^{-(m+n-1)}$, valid for every $\eta>0$, and $\eta^{m+n-1}T(\eta)$ converges, as
$\eta\to0^+$, to the constant
$2^{n}/\bigl((m+n-1)(\max_\R(\psi+\varphi'))^{m+n-1}\bigr)$; to leading order, the
lifespan thus depends on the data only through the maximum of $\psi+\varphi'$. If $\psi+\varphi'\le0$ on $\R$, the solution is global for every amplitude below an
explicit threshold. Unless $\psi+\varphi'\equiv0$, in which case the solution is a
travelling wave, the component $v_t+v_x$ then tends to zero at a rate that depends
neither on the data nor on the amplitude, and the gradient of the solution converges
uniformly to that of a free wave travelling to the right. No sign condition is imposed on $\psi-\varphi'$ in
either case, and the global solutions are not confined to travelling waves: they include
solutions for which the source is active at $t=0$.

The sign of $\psi-\varphi'$ is irrelevant because the source of \eqref{eq:main} is
nonnegative, so that $v\mapsto-v$ is not a symmetry of the equation, in contrast with odd
sources such as $\abs{v_t\pm v_x}^{p-1}(v_t\pm v_x)$. Reversing the signs of the data
therefore does not reverse the conclusion. The exponent $m+n-1$ is the one predicted by
the balance $T\eta^{m+n}\sim\eta$ between the contribution of the source to the first
derivatives of the solution and the size of those derivatives at $t=0$.

The proofs are carried out on the characteristic components $P=v_t+v_x$ and
$Q=v_t-v_x$, which satisfy
$$
(\partial_t-\partial_x)P=\mathcal F(P,Q),
\qquad
(\partial_t+\partial_x)Q=\mathcal F(P,Q),
\qquad
\mathcal F(P,Q)=\abs{P}^{m}\Bigl|\frac{P+Q}{2}\Bigr|^{n}\ \ge0 .
$$
Because the source is nonnegative, $P$ is nondecreasing along the characteristics
$x+t=\mathrm{const}$ and $Q$ is nondecreasing along the characteristics
$x-t=\mathrm{const}$. Along a characteristic of the first family, $P$ solves a scalar
equation whose right-hand side vanishes at $P=0$ and is locally Lipschitz, so that the
zero set of $P$ is invariant along that family. For compactly supported data, $P$, and
with it the whole source, therefore remains in the slab $\{\abs{x+t}\le\rho\}$, which
moves with speed one, although the solution itself spreads through the light cone. A
characteristic of the second family crosses this slab during a time interval of length at
most $\rho$, so that $Q$ cannot accumulate forcing over long times.

When $\psi+\varphi'\le0$, the component $P$ is driven towards $0$ without crossing it and
stays bounded by its initial size, and the time window just described closes a bootstrap
for $Q$ and produces the global solutions. The same localization, combined with the
integrability of the source along each characteristic of the first family, shows that $P$
is absorbed, and a comparison argument then gives the rate of absorption and the limit of the gradient. When $\psi+\varphi'$ is positive at some point, the characteristic issued from
a maximum point of $\psi+\varphi'$ meets, after a waiting time not exceeding $\rho$, only
characteristics of the second family that started outside the support of the data. There
$Q\ge0$, and $P$ obeys the closed inequality $Y'\ge2^{-n}Y^{m+n}$, which gives the upper
bound for the lifespan; this step uses only the monotonicity of $Q$. The localization
enters the blow-up branch through the exact value of the limit: it shows that, after the
waiting time, $Q$ is of order $\eta^{m+n}$ along the selected characteristic and hence
negligible against $P$, so that the upper bound is attained asymptotically.

The paper is organized as follows. Section~\ref{sec:results} states the main result and
illustrates it on simple data. Section~\ref{sec:char} derives the characteristic system
and shows that it is equivalent to \eqref{eq:main}. Section~\ref{sec:wellposed} contains
the local theory and the proof of the lower bound for the lifespan, and
Section~\ref{sec:structure} establishes the localization property. The upper bound and the
limit of $\eta^{m+n-1}T(\eta)$ are proved in Section~\ref{sec:blowup}, and the global
existence and the long-time behaviour in Section~\ref{sec:global}.
Section~\ref{sec:extensions} treats the endpoint exponents $m,n\ge1$ and general sources,
and Section~\ref{sec:conclusion} discusses the hypotheses and states some open questions.
The local existence theorem, which is classical, is proved in Appendix~\ref{app:local}.

Throughout, $\norm{\cdot}_\infty$ denotes the supremum norm over the set specified by the
context and $C_0^k(\R)$ the space of compactly supported functions of class $C^k$. The
constants $c$, $C$, $c_1$, $C_1$, $\kappa_0$, $\kappa_1$ depend only on $m$, $n$, $\rho$
and on the profiles $\varphi,\psi$, never on $\eta$, and we write $A\asymp B$ when
$cB\le A\le CB$ for such constants.

\section{Main results}\label{sec:results}

Write
\begin{equation}\label{eq:P0Q0}
P_0:=\psi+\varphi',
\qquad
Q_0:=\psi-\varphi',
\end{equation}
for the characteristic components of the data, so that $P_0,Q_0\in C^1_0(\R)$, and set
\begin{equation}\label{eq:E0}
E_0:=\norm{P_0}_\infty+\norm{Q_0}_\infty ,
\qquad
P_0^{\max}:=\max_\R P_0 .
\end{equation}
The quantity $E_0$ is positive: $P_0\equiv Q_0\equiv0$ would force $\psi\equiv0$ and
$\varphi'\equiv0$, hence $\varphi\equiv0$ by compactness of the support. Since $P_0$ has
compact support, $P_0^{\max}\ge0$, so that either $P_0^{\max}>0$ or $P_0\le0$ on $\R$. We
also introduce the constants
\begin{equation}\label{eq:constants}
\begin{gathered}
c_1:=\frac{\bigl(1-2^{1-m-n}\bigr)2^{\,n-1}}{(m+n-1)E_0^{\,m+n-1}},
\qquad
C_1:=\frac{2^{\,n}}{(m+n-1)(P_0^{\max})^{m+n-1}},
\\[4pt]
\eta_0:=\frac{1}{E_0}\Bigl(\frac{1}{2^{m}\rho}\Bigr)^{1/(m+n-1)},
\end{gathered}
\end{equation}
the second one being defined when $P_0^{\max}>0$.

\begin{theorem}\label{thm:main}
Assume \eqref{eq:standing} and \eqref{eq:support}.
\begin{enumerate}[label=\textup{(\roman*)},leftmargin=2.4em]
\item If $P_0^{\max}>0$, then for every $\eta>0$,
\begin{equation}\label{eq:two-sided}
c_1\eta^{-(m+n-1)}\ \le\ T(\eta)\ \le\ \rho+C_1\eta^{-(m+n-1)},
\end{equation}
and
\begin{equation}\label{eq:limit}
\lim_{\eta\to0^+}\eta^{m+n-1}\,T(\eta)=C_1 .
\end{equation}
\item If $P_0\le0$ on $\R$, then $T(\eta)=+\infty$ for every $0<\eta<\eta_0$.
\item If $P_0\le0$ on $\R$, $P_0\not\equiv0$ and $0<\eta<\eta_0$, then
\begin{equation}\label{eq:decay}
\lim_{t\to\infty}t^{1/(m+n-1)}\,\norm{(v_t+v_x)(\cdot,t)}_\infty
=\Bigl(\frac{2^{\,n}}{m+n-1}\Bigr)^{1/(m+n-1)} ,
\end{equation}
and there is a continuous function $Q_\infty$ on $\R$, vanishing on $[\rho,\infty)$ and
tending to $0$ at $-\infty$, such that
\begin{equation}\label{eq:gradient-limit}
\lim_{t\to\infty}\,\sup_{x\in\R}\,\bigl|(v_t-v_x)(x,t)-Q_\infty(x-t)\bigr|=0 .
\end{equation}
\end{enumerate}
In particular, for $0<\eta<\eta_0$ the solution blows up in finite time if and only if
$\psi+\varphi'$ is positive at some point.
\end{theorem}

The lower bound in \eqref{eq:two-sided} holds for all data, whatever the sign of $P_0$,
and is proved in Section~\ref{sec:wellposed}. The upper bound and the limit are proved in
Section~\ref{sec:blowup}, and parts (ii) and (iii) in Section~\ref{sec:global}.

By \eqref{eq:decay} and \eqref{eq:gradient-limit}, the gradient $(v_t,v_x)$ converges
uniformly to $\frac12\bigl(Q_\infty(x-t),-Q_\infty(x-t)\bigr)$, the gradient of a free wave
travelling to the right. The case $P_0\equiv0$, excluded in (iii), is the travelling wave
itself: if $\psi=-\varphi'$, the function $v(x,t)=\eta\varphi(x-t)$ satisfies
$v_t+v_x=0$, so that the source vanishes, and it solves \eqref{eq:main} for every
$\eta>0$; by Propositions~\ref{prop:equivalence} and~\ref{prop:wellposed}(a) it is the
solution.

The finite-lifespan branch requires nothing of $Q_0=\psi-\varphi'$. The condition
$Q_0\ge0$, which would make a pointwise comparison immediate, is not assumed: if
$Q_0\le0$ while $P_0$ is positive somewhere, part (i) applies unchanged, and if
$P_0\le0$, parts (ii) and (iii) apply whatever the sign of $Q_0$. Since $\varphi$ has
compact support, $\int_\R\varphi'\,dx=0$, and therefore
\begin{equation}\label{eq:integrals}
\int_\R P_0\,dx=\int_\R Q_0\,dx=\int_\R\psi\,dx .
\end{equation}
Hence part (i) applies whenever $\int_\R\psi\,dx>0$, and under the hypothesis of part
(ii) either $P_0\equiv0$ or $\int_\R\psi\,dx<0$. Part (i) also applies when
$\psi-\varphi'\ge0$: if $Q_0\not\equiv0$ this follows from \eqref{eq:integrals}, and if
$Q_0\equiv0$, then $P_0=2\varphi'$ with $\varphi\not\equiv0$, while a nonzero compactly
supported function of class $C^1$ cannot be nonincreasing on $\R$. For such data the
additive constant $\rho$ in \eqref{eq:two-sided} can be dropped; see
Section~\ref{sec:blowup}.

The limit \eqref{eq:limit} shows that the upper bound is asymptotically sharp, and that,
to leading order, the lifespan depends on the data only through $P_0^{\max}$. The lower
bound $c_1\eta^{-(m+n-1)}$ remains useful as an estimate valid for every $\eta>0$. Both
bounds hold for all $\eta>0$ because the source is positively homogeneous of degree $m+n$
and the estimates used are homogeneous in the amplitude, and the additive constant $\rho$
in the upper bound is absorbed as soon as $\eta^{m+n-1}\le C_1/\rho$.

The same constant $2^{n}/(m+n-1)$ appears in \eqref{eq:limit} and in \eqref{eq:decay}.
In both cases the behaviour of $v_t+v_x$ along the characteristics $x+t=\mathrm{const}$ is
governed, asymptotically, by the scalar equation $y'=2^{-n}\abs{y}^{m+n}$: from a positive
value $y_0$ its solution blows up at time $2^{n}/\bigl((m+n-1)y_0^{m+n-1}\bigr)$, and from
a negative value it is absorbed at the rate $\bigl(2^{n}/((m+n-1)t)\bigr)^{1/(m+n-1)}$. In
particular, the decay rate in \eqref{eq:decay} is the same for all admissible data and
amplitudes.

The global branch is not confined to travelling waves. It contains every pair
$\psi=-\varphi'-h$ with $\varphi\in C^2_0(\R)$ arbitrary and $0\le h\in C^1_0(\R)$
supported in $[-\rho,\rho]$, since then $P_0=-h$. Such a solution has the form
$v(x,t)=A(x-t)$ only if $h\equiv0$, because $v_t+v_x\equiv0$ forces $P_0\equiv0$. At $t=0$
the source equals $\eta^{m+n}h^m\abs{\varphi'+h}^n$, which is not identically zero as soon
as $h(x)>0$ and $\varphi'(x)+h(x)\ne0$ at some point.

The simplest data make all the constants explicit. Let $h\in C^1_0(\R)$ satisfy $h\ge0$,
$h\not\equiv0$ and $\supp h\subset[-\rho,\rho]$. For $\varphi\equiv0$ and $\psi=h$ one has
$P_0=Q_0=h$ and $E_0=2\norm{h}_\infty$, the waiting time is absent, and
$$
\frac{\bigl(1-2^{1-m-n}\bigr)2^{\,n-1}}
{(m+n-1)\bigl(2\norm{h}_\infty\bigr)^{m+n-1}}\,\eta^{-(m+n-1)}
\ \le\ T(\eta)\ \le\
\frac{2^{\,n}}{(m+n-1)\norm{h}_\infty^{\,m+n-1}}\,\eta^{-(m+n-1)}
$$
for every $\eta>0$, the ratio of $T(\eta)$ to the upper bound tending to $1$ as
$\eta\to0^+$. For $\varphi\equiv0$ and $\psi=-h$ one has $P_0=Q_0=-h$. The solution is
then global for $\eta<\bigl(2^{\,m}\rho\bigr)^{-1/(m+n-1)}/\bigl(2\norm{h}_\infty\bigr)$,
it satisfies $-\eta\norm{h}_\infty\le v_t+v_x\le0$ and $\abs{v_t-v_x}<3\eta\norm{h}_\infty$
for all times by \eqref{eq:global-bounds}, the source equals $\eta^{m+n}h^{m+n}$ at $t=0$,
and $\norm{(v_t+v_x)(\cdot,t)}_\infty$ behaves like
$\bigl(2^{\,n}/((m+n-1)t)\bigr)^{1/(m+n-1)}$ as $t\to\infty$, independently of $h$ and
$\eta$. The two examples differ only in the sign of $\psi$; for an equation with an odd
source they would have the same lifespan.

\section{The characteristic system}\label{sec:char}

We pass to the variables adapted to the two characteristic directions, record the
properties of the source that will be used throughout, and show that the resulting
first-order system carries exactly the same information as the Cauchy problem
\eqref{eq:main}.

Let
\begin{equation}\label{eq:PQ-def}
P=v_t+v_x,
\qquad
Q=v_t-v_x,
\qquad\text{so that}\qquad
v_t=\frac{P+Q}{2},
\quad
v_x=\frac{P-Q}{2},
\end{equation}
and
\begin{equation}\label{eq:F-def}
\mathcal F(P,Q):=\abs{P}^m\Bigl|\frac{P+Q}{2}\Bigr|^{\,n}.
\end{equation}
For $v\in C^2$ one has $(\partial_t-\partial_x)P=v_{tt}-v_{xx}$ and
$(\partial_t+\partial_x)Q=v_{tt}-v_{xx}$, so that \eqref{eq:main} takes the form
\begin{equation}\label{eq:system}
\begin{cases}
(\partial_t-\partial_x)P=\mathcal F(P,Q),\\[2pt]
(\partial_t+\partial_x)Q=\mathcal F(P,Q),
\end{cases}
\qquad
P(x,0)=\eta P_0(x),
\quad
Q(x,0)=\eta Q_0(x),
\end{equation}
with $P_0$ and $Q_0$ as in \eqref{eq:P0Q0}. Integrating the two equations along their
respective characteristics yields the system of integral equations
\begin{align}
P(x,t)&=\eta P_0(x+t)+\int_0^t\mathcal F(P,Q)(x+t-s,s)\,ds,\label{eq:P-int}\\
Q(x,t)&=\eta Q_0(x-t)+\int_0^t\mathcal F(P,Q)(x-t+s,s)\,ds,\label{eq:Q-int}
\end{align}
which keep the dependence on $\eta$ explicit and will serve as the basic formulation.

The source has the following elementary properties. For $m,n>1$ the functions
$s\mapsto\abs{s}^m$ and $s\mapsto\abs{s}^n$ are of class $C^1$, so that $\mathcal F$ is of
class $C^1$ on $\R^2$. Writing $V=(P+Q)/2$, one has
$$
\mathcal F_P=m\abs{P}^{m-2}P\,\abs{V}^n+\frac n2\abs{P}^m\abs{V}^{n-2}V,
\qquad
\mathcal F_Q=\frac n2\abs{P}^m\abs{V}^{n-2}V ,
$$
and since $\abs{P}\le R$ and $\abs{V}\le R/2$ whenever $\abs{P}+\abs{Q}\le R$, we obtain,
for such $(P,Q)$,
\begin{equation}\label{eq:F-bounds}
0\le\mathcal F(P,Q)\le\kappa_0R^{m+n},
\qquad
\abs{\mathcal F_P(P,Q)}+\abs{\mathcal F_Q(P,Q)}\le\kappa_1R^{m+n-1},
\end{equation}
where
\begin{equation}\label{eq:kappas}
\kappa_0:=2^{-n},
\qquad
\kappa_1:=2^{-n}(m+2n);
\end{equation}
indeed, $\abs{\mathcal F_P}+\abs{\mathcal F_Q}\le m\abs{P}^{m-1}\abs{V}^n
+n\abs{P}^m\abs{V}^{n-1}\le\bigl(m2^{-n}+n2^{1-n}\bigr)R^{m+n-1}$. The set
$\{\abs{P}+\abs{Q}\le R\}$ being convex, the mean value theorem then gives, for all
$(P_i,Q_i)$ with $\abs{P_i}+\abs{Q_i}\le R$, $i=1,2$,
\begin{equation}\label{eq:F-lip}
\abs{\mathcal F(P_1,Q_1)-\mathcal F(P_2,Q_2)}
\le\kappa_1R^{m+n-1}\bigl(\abs{P_1-P_2}+\abs{Q_1-Q_2}\bigr).
\end{equation}

We now verify that solving \eqref{eq:P-int}--\eqref{eq:Q-int} is equivalent to solving
\eqref{eq:main}. The only point that requires attention is the reconstruction of $v$,
which involves a compatibility condition between the two components of the gradient.

\begin{proposition}\label{prop:equivalence}
Let $T>0$.
\begin{enumerate}[label=\textup{(\alph*)},leftmargin=2.2em]
\item If $v\in C^2(\R\times[0,T])$ solves \eqref{eq:main}, then $(P,Q)$ defined by
\eqref{eq:PQ-def} belongs to $C^1(\R\times[0,T])$ and solves \eqref{eq:system} and
\eqref{eq:P-int}--\eqref{eq:Q-int}.
\item Conversely, if $(P,Q)\in C^1(\R\times[0,T])$ solves
\eqref{eq:P-int}--\eqref{eq:Q-int}, then
\begin{equation}\label{eq:reconstruction}
v(x,t):=\eta\varphi(x)+\int_0^t\frac{P(x,s)+Q(x,s)}{2}\,ds
\end{equation}
belongs to $C^2(\R\times[0,T])$, solves \eqref{eq:main}, and satisfies
$v_t=(P+Q)/2$ and $v_x=(P-Q)/2$.
\end{enumerate}
\end{proposition}

\begin{proof}
(a) Since $v\in C^2$, the functions $P$ and $Q$ are of class $C^1$, and
\eqref{eq:system} was verified above. Setting $\Pi(s):=P(x+t-s,s)$ we obtain
$\Pi'(s)=\bigl[(\partial_t-\partial_x)P\bigr](x+t-s,s)=\mathcal F(P,Q)(x+t-s,s)$, and
integrating over $[0,t]$ gives \eqref{eq:P-int}. The identity \eqref{eq:Q-int} is proved
in the same way along $s\mapsto(x-t+s,s)$.

(b) Differentiating \eqref{eq:P-int}--\eqref{eq:Q-int} along the characteristics shows
that $(P,Q)$ solves \eqref{eq:system}:
\begin{equation}\label{eq:system-again}
(\partial_t-\partial_x)P=\mathcal F,
\qquad
(\partial_t+\partial_x)Q=\mathcal F .
\end{equation}
The function $v$ defined by \eqref{eq:reconstruction} satisfies
\begin{equation}\label{eq:vt}
v_t=\frac{P+Q}{2}\in C^1,
\end{equation}
and, differentiating under the integral sign,
$v_x(x,t)=\eta\varphi'(x)+\frac12\int_0^t(P_x+Q_x)(x,s)\,ds$. Subtracting the two
equations in \eqref{eq:system-again} gives $P_t-Q_t=P_x+Q_x$, so that
\begin{equation}\label{eq:vx}
v_x(x,t)=\eta\varphi'(x)+\frac12\bigl[P(x,t)-Q(x,t)\bigr]
-\frac12\bigl[P(x,0)-Q(x,0)\bigr]=\frac{P-Q}{2},
\end{equation}
because $P(x,0)-Q(x,0)=\eta(P_0-Q_0)=2\eta\varphi'$. Hence $v_t$ and $v_x$ are of class
$C^1$ and $v\in C^2$. Moreover
$$
v_{tt}-v_{xx}
=\frac{(P_t-P_x)+(Q_t+Q_x)}{2}
=\frac{\mathcal F+\mathcal F}{2}
=\mathcal F(P,Q)
=\abs{v_t+v_x}^m\abs{v_t}^n,
$$
by \eqref{eq:PQ-def} and \eqref{eq:F-def}. Finally $v(x,0)=\eta\varphi(x)$ and
$v_t(x,0)=\eta(P_0+Q_0)/2=\eta\psi(x)$.
\end{proof}

From now on we work exclusively with the pair $(P,Q)$, and $T(\eta)$ also denotes the
lifespan of the $C^1$ solution of \eqref{eq:P-int}--\eqref{eq:Q-int}. By
Proposition~\ref{prop:equivalence} this is the same quantity as \eqref{eq:lifespan-def}.
The source $\mathcal F$ is positively homogeneous of degree $m+n$, which reflects the
invariance of the equation under the scaling
$v\mapsto\lambda^{-(m+n-2)/(m+n-1)}v(\lambda x,\lambda t)$. This scaling dilates the
profiles together with the amplitude, so it does not by itself determine $T(\eta)$ for a
fixed pair $(\varphi,\psi)$. It explains, however, why the explicit bounds below, measured
in units of $\rho$, and the threshold $\eta_0$ involve the data only through the
scale-invariant quantities $\rho\,(\eta E_0)^{m+n-1}$ and $\rho\,(\eta P_0^{\max})^{m+n-1}$.

\section{Local theory and the lower bound}\label{sec:wellposed}

We first collect uniqueness, finite speed of propagation, local existence and the
continuation criterion, and then prove the lower bound in \eqref{eq:two-sided}. For $K>0$
put
\begin{equation}\label{eq:tau}
\tau(K):=\frac{1}{2^{m+n+1}\kappa_1K^{m+n-1}} .
\end{equation}

\begin{proposition}\label{prop:wellposed}
\begin{enumerate}[label=\textup{(\alph*)},leftmargin=2.2em]
\item Let $t_0\ge0$ and $T>t_0$, and let $(P_i,Q_i)$, $i=1,2$, be continuous solutions on
$\R\times[t_0,T]$ of
\begin{equation}\label{eq:shifted}
\begin{aligned}
P(x,t)&=p(x+t-t_0)+\int_{t_0}^t\mathcal F(P,Q)(x+t-s,s)\,ds,\\
Q(x,t)&=q(x-t+t_0)+\int_{t_0}^t\mathcal F(P,Q)(x-t+s,s)\,ds,
\end{aligned}
\end{equation}
with continuous data $(p_i,q_i)$. If $p_1=p_2$ and $q_1=q_2$ on an interval $[a,b]$, then
$(P_1,Q_1)=(P_2,Q_2)$ on the triangle
$\Delta:=\{(x,t):\ t_0\le t\le T,\ a+(t-t_0)\le x\le b-(t-t_0)\}$. In particular,
\eqref{eq:P-int}--\eqref{eq:Q-int} has at most one continuous solution on
$\R\times[0,T]$, and this solution vanishes where $\abs{x}>t+\rho$.
\item Let $t_0\ge0$, let $p,q\in C^1(\R)$ be bounded, with bounded and uniformly
continuous derivatives, and let $K>0$ satisfy $\norm{p}_\infty+\norm{q}_\infty\le K$.
Then \eqref{eq:shifted} has a unique continuous solution on $\R\times[t_0,t_0+\tau(K)]$,
and it satisfies $\norm{P}_\infty+\norm{Q}_\infty\le2K$. Moreover, every bounded continuous
solution of \eqref{eq:shifted} on a strip $\R\times[t_0,T']$ is of class $C^1$.
\item If $T(\eta)<\infty$, then
$$
\limsup_{t\to T(\eta)^-}
\bigl(\norm{P(\cdot,t)}_\infty+\norm{Q(\cdot,t)}_\infty\bigr)=+\infty .
$$
\end{enumerate}
\end{proposition}

\begin{proof}
(a) Let $R$ be the maximum of $\abs{P_i}+\abs{Q_i}$, $i=1,2$, over the compact set
$\Delta$, write $\Delta_\tau:=\{x:\ (x,\tau)\in\Delta\}$, and let
$\omega(\tau):=\max_{x\in\Delta_\tau}\bigl(\abs{P_1-P_2}+\abs{Q_1-Q_2}\bigr)(x,\tau)$ for
those $\tau$ with $\Delta_\tau\ne\varnothing$. For $(x,\tau)\in\Delta$ and
$t_0\le s\le\tau$ one has $a+(s-t_0)\le x+\tau-s\le b-(s-t_0)$ and
$a+(s-t_0)\le x-\tau+s\le b-(s-t_0)$, so both characteristics through $(x,\tau)$ remain in
$\Delta$ and start in $[a,b]$, where the data coincide. Subtracting the integral equations
and using \eqref{eq:F-lip},
$$
\bigl(\abs{P_1-P_2}+\abs{Q_1-Q_2}\bigr)(x,\tau)
\le2\kappa_1R^{m+n-1}\int_{t_0}^\tau\omega(s)\,ds ,
$$
and Gronwall's inequality gives $\omega\equiv0$. Uniqueness on $\R\times[0,T]$ follows by
taking arbitrarily large intervals $[a,b]$, and the support property by comparing with the
zero solution on intervals $[a,b]$ contained in $(\rho,\infty)$ or in $(-\infty,-\rho)$,
where the data of \eqref{eq:system} vanish.

(c) Suppose on the contrary that $T_*:=T(\eta)<\infty$ and that
$\norm{P(\cdot,t)}_\infty+\norm{Q(\cdot,t)}_\infty\le K$ for all $t\in[0,T_*)$, for some
$K>0$. Choose $t_0\in\bigl(T_*-\tau(K)/2,T_*\bigr)$. The functions $p:=P(\cdot,t_0)$ and
$q:=Q(\cdot,t_0)$ are of class $C^1$ and, by (a), have compact support, so they are
bounded with bounded and uniformly continuous derivatives, and
$\norm{p}_\infty+\norm{q}_\infty\le K$. By (b), \eqref{eq:shifted} has a solution on
$\R\times[t_0,t_0+\tau(K)]$, and $t_0+\tau(K)>T_*+\tau(K)/2$. Splitting the integrals at
$t_0$ shows that the pair obtained by concatenation is bounded, continuous and satisfies
\eqref{eq:P-int}--\eqref{eq:Q-int} on $\R\times[0,T_*+\tau(K)/2]$; by (b) it is of class
$C^1$, and by Proposition~\ref{prop:equivalence} it yields a classical solution of
\eqref{eq:main} beyond $T_*$, contradicting the definition of $T(\eta)$.

Part (b) is classical, and its proof is given in Appendix~\ref{app:local}.
\end{proof}

Only the supremum norm of the solution enters (c), because the existence time in (b) does
not depend on the derivatives of the data. The lower bound for the lifespan now follows
from a comparison with a scalar equation.

\begin{proof}[Proof of the lower bound in \eqref{eq:two-sided}]
We show that $T(\eta)\ge c_1\eta^{-(m+n-1)}$ for all data satisfying \eqref{eq:standing}
and \eqref{eq:support}. Put $\alpha:=m+n$ and
$K(t):=\norm{P(\cdot,t)}_\infty+\norm{Q(\cdot,t)}_\infty$ for $0\le t<T(\eta)$. This
function is finite and continuous, because $P$ and $Q$ are continuous and vanish for
$\abs{x}>t+\rho$; for the same reason each of the two norms is continuous in $t$. By
\eqref{eq:P-int}--\eqref{eq:Q-int} and \eqref{eq:F-bounds},
\begin{equation}\label{eq:apriori}
K(t)\le W(t):=\eta E_0+2\kappa_0\int_0^tK(s)^{\alpha}\,ds,
\qquad 0\le t<T(\eta).
\end{equation}
The function $W$ is positive and of class $C^1$, with
$W'=2\kappa_0K^{\alpha}\le2\kappa_0W^{\alpha}$, so that
$(W^{1-\alpha})'\ge-2\kappa_0(\alpha-1)$ and
$$
W(t)^{1-\alpha}\ge(\eta E_0)^{1-\alpha}-2\kappa_0(\alpha-1)t .
$$
Let
$$
t_*:=\frac{\bigl(1-2^{1-\alpha}\bigr)(\eta E_0)^{1-\alpha}}{2\kappa_0(\alpha-1)}
=c_1\eta^{-(m+n-1)} .
$$
For $t<\min\{t_*,T(\eta)\}$ the right-hand side above exceeds
$2^{1-\alpha}(\eta E_0)^{1-\alpha}$, so that $K(t)\le W(t)<2\eta E_0$. If $T(\eta)\le t_*$,
then $K$ would be bounded on $[0,T(\eta))$, contradicting
Proposition~\ref{prop:wellposed}(c). Hence $T(\eta)>t_*$.
\end{proof}

\section{Characteristic localization}\label{sec:structure}

This section contains the structural property on which the paper rests. The zero set of
$P$ is invariant along the characteristics $x+t=\mathrm{const}$. For compactly supported
data this confines $P$ and the source to a slab of fixed width, and therefore limits the
time during which the second component can be forced.

We fix $\eta>0$ and write
$$
Y_\xi(t):=P(\xi-t,t),
\qquad
q_\xi(t):=Q(\xi-t,t),
\qquad 0\le t<T(\eta),
$$
for the values of $P$ and $Q$ along the characteristic $x+t=\xi$. By \eqref{eq:P-int} with
$x=\xi-t$, one has $Y_\xi(t)=\eta P_0(\xi)+\int_0^t\mathcal F(P,Q)(\xi-s,s)\,ds$ with a
continuous integrand, so that $Y_\xi$ is of class $C^1$ and
\begin{equation}\label{eq:Y-ODE}
Y_\xi'(t)=\abs{Y_\xi(t)}^m\Bigl|\frac{Y_\xi(t)+q_\xi(t)}{2}\Bigr|^{\,n},
\qquad
Y_\xi(0)=\eta P_0(\xi).
\end{equation}
Since $\mathcal F\ge0$, \eqref{eq:P-int}--\eqref{eq:Q-int} also show that $P$ is
nondecreasing along the characteristics $x+t=\mathrm{const}$, and $Q$ along the
characteristics $x-t=\mathrm{const}$:
\begin{equation}\label{eq:monotone}
Y_\xi(t)\ge\eta P_0(\xi),
\qquad
Q(\zeta+t,t)\ge\eta Q_0(\zeta)
\qquad\text{for all }\xi,\zeta\in\R,\ 0\le t<T(\eta).
\end{equation}

\begin{lemma}\label{lem:localization}
If $P_0(\xi)=0$, then $Y_\xi\equiv0$ on $[0,T(\eta))$. In particular
\begin{equation}\label{eq:slab}
\supp P(\cdot,t)\subset\supp P_0-t\subset[-\rho-t,\rho-t],
\qquad
\supp\mathcal F(P,Q)(\cdot,t)\subset[-\rho-t,\rho-t],
\end{equation}
for $0\le t<T(\eta)$.
\end{lemma}

\begin{proof}
Fix $T'<T(\eta)$ and $\xi\in\R$, let
$R:=\sup_{\R\times[0,T']}\bigl(\abs{P}+\abs{Q}\bigr)$, and put
$$
f(t,y):=\abs{y}^m\Bigl|\frac{y+q_\xi(t)}{2}\Bigr|^{\,n},
\qquad t\in[0,T'],\ \abs{y}\le R.
$$
Then $f$ is continuous, locally Lipschitz in $y$, and $f(t,0)=0$. By \eqref{eq:Y-ODE},
$Y_\xi$ solves $y'=f(t,y)$ with $y(0)=\eta P_0(\xi)=0$, as does $y\equiv0$. Uniqueness
for the scalar equation gives $Y_\xi\equiv0$ on $[0,T']$, and $T'$ is arbitrary. The first
inclusion in \eqref{eq:slab} follows because $P(x,t)=Y_{x+t}(t)$, and the second because
$\mathcal F(P,Q)$ vanishes where $P$ does.
\end{proof}

The invariance of the zero set requires less than the Lipschitz continuity of $f$ in $y$:
it suffices that $\abs{f(t,y)}\le C\abs{y}$ for $t\in[0,T']$ and $\abs{y}\le R$, since
Gronwall's inequality applied to $\abs{Y_\xi(t)}\le C\int_0^t\abs{Y_\xi(s)}\,ds$ then forces
$Y_\xi\equiv0$. This bound holds as soon as $m\ge1$, whatever $n>0$, and this will matter
in Section~\ref{sec:extensions}. For $0<m<1$ the right-hand side of \eqref{eq:Y-ODE} is
not Lipschitz in $y$ at $y=0$ when the second factor does not vanish, uniqueness for the
scalar equation may fail, and $P$ can leave the slab.

The same uniqueness argument controls the characteristics on which $P$ starts
nonpositive. Let $P_0(\xi)\le0$. If $Y_\xi(t_2)>0$ for some $t_2$, then $Y_\xi(t_1)=0$ for
some $t_1\le t_2$, and uniqueness for the scalar equation with initial time $t_1$ gives
$Y_\xi\equiv0$ on $[t_1,T(\eta))$, which is absurd. If $P_0(\xi)<0$, then $Y_\xi$ cannot
vanish at a time $t_1>0$, since uniqueness backward in time from $t_1$ would give
$Y_\xi(0)=0$. Together with \eqref{eq:monotone}, this shows that, for $0\le t<T(\eta)$,
\begin{equation}\label{eq:signs}
\eta P_0(\xi)\le Y_\xi(t)\le0\quad\text{if }P_0(\xi)\le0,
\qquad
Y_\xi(t)<0\quad\text{if }P_0(\xi)<0 .
\end{equation}
In particular $\abs{Y_\xi}$ is nonincreasing when $P_0(\xi)\le0$, and if $P_0\le0$ on $\R$,
then
\begin{equation}\label{eq:P-bounded}
-\eta\norm{P_0}_\infty\le P(x,t)\le0
\qquad\text{on }\R\times[0,T(\eta)).
\end{equation}

The solution itself spreads through the light cone $\abs{x}\le t+\rho$, but $P$ travels
to the left inside the slab $\{\abs{x+t}\le\rho\}$ of width $2\rho$, and the source vanishes
outside it. A characteristic $x-t=\zeta$ of the second family meets this slab only while
$\abs{\zeta+2s}\le\rho$, that is, during a time interval of length at most $\rho$. This
bounds the forcing that $Q$ can accumulate.

\begin{lemma}\label{lem:window}
For every $(x,t)\in\R\times[0,T(\eta))$,
$$
Q(x,t)=\eta Q_0(x-t)
+\int_{I(x,t)}\mathcal F(P,Q)(x-t+s,s)\,ds,
\qquad
I(x,t):=\bigl\{s\in[0,t]:\ \abs{x-t+2s}\le\rho\bigr\},
$$
and $\abs{I(x,t)}\le\rho$; moreover $I(x,t)\subset[t-\rho,t]$ whenever $x+t\le\rho$.
Consequently
\begin{equation}\label{eq:window}
\abs{Q(x,t)}\le\eta\norm{Q_0}_\infty
+\rho\,\kappa_0\sup_{0\le s\le t}
\bigl(\norm{P(\cdot,s)}_\infty+\norm{Q(\cdot,s)}_\infty\bigr)^{m+n}.
\end{equation}
\end{lemma}

\begin{proof}
By \eqref{eq:slab}, the integrand in \eqref{eq:Q-int} vanishes unless
$\abs{(x-t+s)+s}\le\rho$, that is, unless $s\in I(x,t)$. The condition
$\abs{x-t+2s}\le\rho$ confines $s$ to the interval
$\bigl[(t-x-\rho)/2,(t-x+\rho)/2\bigr]$ of length $\rho$, whose left end point is at least
$t-\rho$ when $x+t\le\rho$, and $I(x,t)\subset[0,t]$ by definition. The estimate
\eqref{eq:window} then follows from \eqref{eq:F-bounds}.
\end{proof}

Lemma~\ref{lem:window} is the source of the smallness condition in part (ii) of
Theorem~\ref{thm:main}: the time window has the fixed length $\rho$, whereas the forcing
of $Q$ during that window is controlled by the amplitude. It is also what makes $Q$
negligible against $P$ in the proofs of \eqref{eq:limit} and of part (iii).

\section{Finite lifespan}\label{sec:blowup}

In this section we assume $P_0^{\max}>0$ and prove the upper bound in
\eqref{eq:two-sided} and the limit \eqref{eq:limit}. The argument follows the
characteristic issued from a point where $P_0$ is positive. On that line $P$ is positive
and nondecreasing, and the only missing information is the sign of $Q$, which enters the
source through the velocity $(P+Q)/2$. This sign is supplied after a waiting time
controlled by the support of the data, using only the monotonicity of $Q$ along its
characteristics, and it replaces the hypothesis $Q_0\ge0$. The localization of
Section~\ref{sec:structure} enters in the proof of the limit.

Let $x_0\in[-\rho,\rho]$ and $t_1:=(x_0+\rho)/2\in[0,\rho]$. For $t\ge t_1$, the
characteristic of the second family through $(x_0-t,t)$ starts at $x_0-2t\le-\rho$, where
$Q_0$ vanishes, being continuous with support in $[-\rho,\rho]$. Hence
\eqref{eq:monotone} gives
\begin{equation}\label{eq:Q-nonneg-late}
Q(x_0-t,t)\ge\eta Q_0(x_0-2t)=0
\qquad\text{for }t_1\le t<T(\eta).
\end{equation}

\begin{proof}[Proof of the upper bound in \eqref{eq:two-sided}]
Let $x_0$ be a maximum point of $P_0$. Since $P_0(x_0)>0$ and
$\supp P_0\subset[-\rho,\rho]$, the point $x_0$ lies in $(-\rho,\rho)$. Write
$Y:=Y_{x_0}$, $q:=q_{x_0}$ and $t_1:=(x_0+\rho)/2$. By \eqref{eq:monotone},
$Y(t)\ge\eta P_0^{\max}>0$ for all $t$, and by \eqref{eq:Q-nonneg-late}, $q(t)\ge0$ for
$t\ge t_1$. Hence $(Y+q)/2\ge Y/2>0$ for $t\ge t_1$, and \eqref{eq:Y-ODE} gives
\begin{equation}\label{eq:Y-exact}
Y'(t)=Y(t)^m\Bigl(\frac{Y(t)+q(t)}{2}\Bigr)^{n}\ge2^{-n}Y(t)^{m+n},
\qquad t\in[t_1,T(\eta)).
\end{equation}
For $t\in[t_1,T(\eta))$, dividing \eqref{eq:Y-exact} by $Y^{m+n}$ and integrating over
$[t_1,t]$ gives
$$
2^{-n}(m+n-1)(t-t_1)\le Y(t_1)^{1-m-n}-Y(t)^{1-m-n}<Y(t_1)^{1-m-n}
\le(\eta P_0^{\max})^{1-m-n},
$$
that is, $t-t_1<C_1\eta^{-(m+n-1)}$. Since $t\in[t_1,T(\eta))$ is arbitrary,
\begin{equation}\label{eq:upper}
T(\eta)\le t_1+C_1\eta^{-(m+n-1)},
\end{equation}
and $t_1\le\rho$ gives the upper bound in \eqref{eq:two-sided}.
\end{proof}

If $Q_0\ge0$ on $\R$, then \eqref{eq:monotone} gives $Q\ge0$ everywhere, the argument
applies with $t_1=0$, and $T(\eta)\le C_1\eta^{-(m+n-1)}$: the waiting time disappears. In
general, the argument works along the characteristic through any point where $P_0$ is
positive, and optimizing over that point gives
$$
T(\eta)\le\inf_{\{x:\,P_0(x)>0\}}
\Bigl\{\frac{x+\rho}{2}
+\frac{2^{\,n}}{m+n-1}\bigl(\eta P_0(x)\bigr)^{-(m+n-1)}\Bigr\},
$$
in which the second term dominates as $\eta\to0^+$, so that a maximum point of $P_0$ is
the best choice at the leading scale. The proof also quantifies the growth: integrating
\eqref{eq:Y-exact} from $t_1$ gives
$$
P(x_0-t,t)\ge\Bigl(\bigl(\eta P_0^{\max}\bigr)^{1-m-n}-2^{-n}(m+n-1)(t-t_1)\Bigr)^{-1/(m+n-1)},
\qquad t_1\le t<T(\eta),
$$
so that $P$ becomes unbounded along the selected characteristic if the solution survives
up to the time $t_1+C_1\eta^{-(m+n-1)}$. The proof shows that the classical solution
cannot persist beyond this time, but it locates no singularity: the solution might cease
to be classical earlier, elsewhere, and in a different way, which is not excluded.

We now show that this upper bound is asymptotically sharp. This
is where the localization of Section~\ref{sec:structure} enters the blow-up branch: after
the waiting time, Lemma~\ref{lem:window} makes $Q$ negligible against $P$ along every
characteristic of the first family, uniformly on the time scale $\eta^{-(m+n-1)}$.

\begin{proof}[Proof of the limit \eqref{eq:limit}]
Put $\alpha:=m+n$ and $B:=\norm{Q_0}_\infty$. By \eqref{eq:upper}, $\eta^{\alpha-1}T(\eta)\le\rho\,\eta^{\alpha-1}+C_1$, so that
$\limsup_{\eta\to0^+}\eta^{\alpha-1}T(\eta)\le C_1$, and it remains to prove the
corresponding lower bound.

For $y_0>0$ let
$$
\bar y(\sigma;y_0):=\bigl(y_0^{1-\alpha}-(\alpha-1)2^{-n}\sigma\bigr)^{-1/(\alpha-1)},
\qquad 0\le\sigma<\frac{2^{\,n}y_0^{1-\alpha}}{\alpha-1},
$$
be the solution of $\bar y'=2^{-n}\bar y^{\alpha}$ with $\bar y(0)=y_0$. Its existence
time equals $C_1$ for $y_0=P_0^{\max}$ and depends continuously on $y_0$. Fix
$\varepsilon\in(0,C_1)$, choose $\delta>0$ such that the existence time for
$y_0=P_0^{\max}+\delta$ exceeds $C_1-\varepsilon/2$, and put
$$
M:=\bar y\bigl(C_1-\varepsilon;P_0^{\max}+\delta\bigr),
\qquad
M':=\max\{M,\norm{P_0}_\infty\},
\qquad
R:=M'+B+2 .
$$
Then $M\ge P_0^{\max}+\delta$. Let $\eta_\varepsilon>0$ be such that
\begin{equation}\label{eq:eps-small}
\rho\,2^{-n}R^{\alpha}\eta^{\alpha-1}\le\min\Bigl\{\frac12,\frac\delta2\Bigr\}
\qquad\text{for }0<\eta\le\eta_\varepsilon .
\end{equation}
Fix $\eta\in(0,\eta_\varepsilon]$, put $T_\eta:=(C_1-\varepsilon)\eta^{1-\alpha}$ and
$T':=\min\{T(\eta),T_\eta\}$, and let $t^*$ be the supremum of the times $t<T'$ such that
\begin{equation}\label{eq:bootstrap}
\norm{P(\cdot,s)}_\infty\le\eta(M'+1)
\quad\text{and}\quad
\norm{Q(\cdot,s)}_\infty\le\eta(B+1)
\qquad\text{for }0\le s\le t .
\end{equation}
The two norms are continuous functions of time, as observed in the proof of the lower bound in Section~\ref{sec:wellposed}, and at $t=0$ they equal $\eta\norm{P_0}_\infty\le\eta M'$ and
$\eta B$, so that $t^*>0$. On $[0,t^*)$ we have $\abs{P}+\abs{Q}\le\eta R$, hence
$\norm{\mathcal F(P,Q)(\cdot,s)}_\infty\le2^{-n}(\eta R)^{\alpha}$ by \eqref{eq:F-bounds},
and we now improve both bounds in \eqref{eq:bootstrap}.

By \eqref{eq:window} and \eqref{eq:eps-small},
$\norm{Q(\cdot,t)}_\infty\le\eta B+\rho\,2^{-n}(\eta R)^{\alpha}\le\eta\bigl(B+\frac12\bigr)$
for $0\le t<t^*$.

Let $\xi\in\R$. If $P_0(\xi)\le0$, then $\abs{Y_\xi}\le\eta\norm{P_0}_\infty\le\eta M'$ by \eqref{eq:signs}. Let $P_0(\xi)>0$, so that $\abs{\xi}<\rho$, and put
$t_\xi:=(\xi+\rho)/2\le\rho$. For $0\le t<t^*$ with $t\le t_\xi$, \eqref{eq:Y-ODE} and
\eqref{eq:eps-small} give
$$
Y_\xi(t)\le\eta P_0(\xi)+\rho\,2^{-n}(\eta R)^{\alpha}
\le\eta\Bigl(P_0^{\max}+\frac\delta2\Bigr)\le\eta M .
$$
For $t_\xi\le t<t^*$ we have $\xi-2t\le-\rho$, hence $Q_0(\xi-2t)=0$, and
Lemma~\ref{lem:window} together with \eqref{eq:eps-small} gives
$$
0\le q_\xi(t)\le\rho\,2^{-n}(\eta R)^{\alpha}\le\frac{\eta\delta}{2} .
$$
Since $Y_\xi>0$, it follows from \eqref{eq:Y-ODE} that
$$
Y_\xi'(t)=Y_\xi^{m}\Bigl(\frac{Y_\xi+q_\xi}{2}\Bigr)^{n}
\le2^{-n}\Bigl(Y_\xi+\frac{\eta\delta}{2}\Bigr)^{\alpha},
\qquad t_\xi\le t<t^* .
$$
The function
$Z(\sigma):=\eta^{-1}\bigl(Y_\xi(t_\xi+\eta^{1-\alpha}\sigma)+\eta\delta/2\bigr)$
therefore satisfies $Z'\le2^{-n}Z^{\alpha}$ and $Z(0)\le P_0^{\max}+\delta$, and the
comparison principle for scalar equations gives
$Z(\sigma)\le\bar y(\sigma;P_0^{\max}+\delta)$. For $t_\xi\le t<t^*$ we have
$\sigma=\eta^{\alpha-1}(t-t_\xi)<\eta^{\alpha-1}T_\eta=C_1-\varepsilon$, hence
$Y_\xi(t)\le\eta Z(\sigma)\le\eta M$. Altogether $\norm{P(\cdot,t)}_\infty\le\eta M'$ for
$0\le t<t^*$.

Both bounds in \eqref{eq:bootstrap} thus hold on $[0,t^*)$ with margins $\eta$ and
$\eta/2$. If $t^*<T'$, continuity would extend \eqref{eq:bootstrap} beyond $t^*$, which is
impossible; hence $t^*=T'$. If $T(\eta)\le T_\eta$, the solution would then be bounded on
$[0,T(\eta))$, in contradiction with Proposition~\ref{prop:wellposed}(c). Therefore
$T(\eta)>(C_1-\varepsilon)\eta^{1-\alpha}$ for every $\eta\in(0,\eta_\varepsilon]$, and
$\liminf_{\eta\to0^+}\eta^{\alpha-1}T(\eta)\ge C_1-\varepsilon$. Since $\varepsilon$ was
arbitrary, \eqref{eq:limit} follows.
\end{proof}

In the one-sided case $Q_0\ge0$, the bound $T(\eta)\le C_1\eta^{-(m+n-1)}$ is therefore attained asymptotically, and so is \eqref{eq:upper} in general.
Without Lemma~\ref{lem:window}, the forcing accumulated by $Q$ along the selected
characteristic could only be bounded by
$t\sup_s\norm{\mathcal F(P,Q)(\cdot,s)}_\infty$, which is of order $\eta$ on the time
scale $\eta^{-(m+n-1)}$ and is not negligible against $P$.

\section{Global existence and long-time behaviour}\label{sec:global}

We now prove parts (ii) and (iii) of Theorem~\ref{thm:main}. Throughout this section we
assume
\begin{equation}\label{eq:negativity}
P_0=\psi+\varphi'\le0
\qquad\text{on }\R .
\end{equation}
By \eqref{eq:P-bounded}, $P$ then stays between $-\eta\norm{P_0}_\infty$ and $0$ for as
long as the solution exists, and Lemma~\ref{lem:window} controls $Q$ provided the
amplitude is small.

\begin{proof}[Proof of Theorem~\ref{thm:main}\textup{(ii)}]
Let $0<\eta<\eta_0$ and put $M:=\eta\bigl(\norm{P_0}_\infty+2\norm{Q_0}_\infty\bigr)$, so
that $\eta\norm{P_0}_\infty+M=2\eta E_0$ and $M-\eta\norm{Q_0}_\infty=\eta E_0>0$. We claim
that $\norm{Q(\cdot,t)}_\infty<M$ for $0\le t<T(\eta)$. This holds at $t=0$, where the norm
equals $\eta\norm{Q_0}_\infty$. Otherwise, the norm being continuous in $t$, there is a
first time $t_2<T(\eta)$ with $\norm{Q(\cdot,t_2)}_\infty=M$. On $[0,t_2]$ we have
$\norm{P(\cdot,s)}_\infty+\norm{Q(\cdot,s)}_\infty\le2\eta E_0$ by \eqref{eq:P-bounded},
and \eqref{eq:window} gives
$$
M\le\eta\norm{Q_0}_\infty+\rho\,\kappa_0(2\eta E_0)^{m+n}
=\eta\norm{Q_0}_\infty+2^{m}\rho\,(\eta E_0)^{m+n}
<\eta\norm{Q_0}_\infty+\eta E_0=M ,
$$
where the strict inequality is $2^{m}\rho\,(\eta E_0)^{m+n-1}<1$, that is, $\eta<\eta_0$.
This contradiction proves the claim. The solution therefore remains bounded, and
Proposition~\ref{prop:wellposed}(c) gives $T(\eta)=+\infty$. We have also shown that
\begin{equation}\label{eq:global-bounds}
-\eta\norm{P_0}_\infty\le P(x,t)\le0,
\qquad
\abs{Q(x,t)}<\eta\bigl(\norm{P_0}_\infty+2\norm{Q_0}_\infty\bigr)
\qquad\text{on }\R\times[0,\infty),
\end{equation}
and, by \eqref{eq:slab}, $P(\cdot,t)$ is supported in $\supp P_0-t$.
\end{proof}

We do not know whether the smallness condition $\eta<\eta_0$ is necessary. The argument
uses it only through the last inequality: the time window of Lemma~\ref{lem:window} has
the length $\rho$ fixed by the data, whereas the forcing of $Q$ during that window grows
superlinearly with the amplitude. Whether blow-up can occur for large $\eta$ under
\eqref{eq:negativity} is an open question.

We now turn to the long-time behaviour. The first step uses only that the solution is
global and that $P_0\le0$.

\begin{lemma}\label{lem:absorption}
Assume \eqref{eq:standing}, \eqref{eq:support}, \eqref{eq:negativity} and
$T(\eta)=+\infty$. Then
\begin{equation}\label{eq:source-integrable-xi}
\int_0^\infty\mathcal F(P,Q)(\xi-s,s)\,ds\le-\eta P_0(\xi)
\qquad\text{for every }\xi\in\R ,
\end{equation}
and $\norm{P(\cdot,t)}_\infty\to0$ as $t\to\infty$.
\end{lemma}

\begin{proof}
By \eqref{eq:signs}, $Y_\xi\le0$, and integrating \eqref{eq:Y-ODE} over $[0,T]$ gives
$\int_0^T\mathcal F(P,Q)(\xi-s,s)\,ds=Y_\xi(T)-\eta P_0(\xi)\le-\eta P_0(\xi)$ for every
$T>0$, which is \eqref{eq:source-integrable-xi}. Write
$\mathcal G(\xi,s):=\mathcal F(P,Q)(\xi-s,s)$ for the source along the characteristic
$x+t=\xi$; by Lemma~\ref{lem:localization} it vanishes unless $\abs{\xi}\le\rho$.

Fix $\xi$ with $P_0(\xi)<0$; for the other values of $\xi$ we have $Y_\xi\equiv0$ by
Lemma~\ref{lem:localization}. By \eqref{eq:monotone} and \eqref{eq:signs}, $Y_\xi$ is nondecreasing and nonpositive, so it has a limit $L\in[\eta P_0(\xi),0]$; suppose that $L<0$. Put
$t_\xi:=(\xi+\rho)/2$ and $q(t):=Q(\xi-t,t)$. For $t\ge t_\xi$ we have $\xi-2t\le-\rho$, so
that $Q_0(\xi-2t)=0$, and the substitution $\xi'=\xi-2t+2s$ in \eqref{eq:Q-int} gives
$$
0\le q(t)=\int_0^t\mathcal G(\xi-2t+2s,s)\,ds
=\frac12\int_{-\rho}^{\xi}\mathcal G\Bigl(\xi',\,t-\frac{\xi-\xi'}{2}\Bigr)\,d\xi' ,
$$
where $t-(\xi-\xi')/2\ge t-t_\xi\ge0$. By Tonelli's theorem and
\eqref{eq:source-integrable-xi} applied on each characteristic $x+t=\xi'$,
$$
\int_{t_\xi}^\infty q(t)\,dt
\le\frac12\int_{-\rho}^{\xi}\int_0^\infty\mathcal G(\xi',s)\,ds\,d\xi'
\le\rho\,\eta\norm{P_0}_\infty .
$$
Hence the set $S:=\{t\ge t_\xi:\ q(t)>\abs{L}/2\}$ has finite measure. For $t\ge t_\xi$
outside $S$, we have $\abs{Y_\xi(t)}\ge\abs{L}$, because $Y_\xi$ is nondecreasing with
limit $L$, and $0\le q(t)\le\abs{L}/2$, so that $\abs{Y_\xi(t)+q(t)}\ge\abs{L}/2$ and
$$
\mathcal G(\xi,t)=\abs{Y_\xi(t)}^m\Bigl|\frac{Y_\xi(t)+q(t)}{2}\Bigr|^n
\ge\abs{L}^m\Bigl(\frac{\abs{L}}{4}\Bigr)^n .
$$
Since $[t_\xi,\infty)\setminus S$ has infinite measure, this contradicts
\eqref{eq:source-integrable-xi}. Therefore $L=0$ for every $\xi$.

Since each $\abs{Y_\xi}$ is nonincreasing, so is $\norm{P(\cdot,t)}_\infty$, and it
suffices to show that $\norm{P(\cdot,k)}_\infty\to0$ along the integers $k\to\infty$. The
functions $\xi\mapsto\abs{Y_\xi(k)}=-P(\xi-k,k)$ are continuous on $[-\rho,\rho]$,
nonincreasing in $k$, and tend pointwise to $0$. By Dini's theorem they converge to $0$
uniformly on $[-\rho,\rho]$, and since $P(x,t)=0$ for $\abs{x+t}>\rho$, the lemma follows.
\end{proof}

\begin{proof}[Proof of Theorem~\ref{thm:main}\textup{(iii)}]
Put $\alpha:=m+n$ and define
\begin{equation}\label{eq:Qinfty}
Q_\infty(\zeta):=\eta Q_0(\zeta)+\int_0^\infty\mathcal F(P,Q)(\zeta+s,s)\,ds ,
\qquad\zeta\in\R .
\end{equation}
Since $v_t+v_x=P$ and $v_t-v_x=Q$, part (iii) amounts to \eqref{eq:decay}, the stated
properties of $Q_\infty$, and the estimate
\begin{equation}\label{eq:Q-limit}
\sup_{x\in\R}\abs{Q(x,t)-Q_\infty(x-t)}
\le\rho\sup_{s\ge t}\norm{\mathcal F(P,Q)(\cdot,s)}_\infty ,
\qquad t\ge0,
\end{equation}
together with the fact that its right-hand side tends to $0$ as $t\to\infty$.

By part (ii) and \eqref{eq:global-bounds}, the solution is global and $\abs{P}+\abs{Q}\le2\eta E_0$ on
$\R\times[0,\infty)$. Let $p(t):=\norm{P(\cdot,t)}_\infty$ and
$\bar q(t):=\sup\{\abs{Q(x,t)}:\ \abs{x+t}\le\rho\}$. By \eqref{eq:signs}, each $\abs{Y_\xi}$ is nonincreasing, so that $p$ is nonincreasing; $p$ is continuous, and it is
positive because $Y_\xi(t)<0$ for all $t$ when $P_0(\xi)<0$. By
Lemma~\ref{lem:absorption}, $p(t)\to0$. Since $\mathcal F(P,Q)$ vanishes outside the slab
$\{\abs{x+t}\le\rho\}$, \eqref{eq:F-def} gives
\begin{equation}\label{eq:F-slab}
\norm{\mathcal F(P,Q)(\cdot,s)}_\infty\le p(s)^m\Bigl(\frac{p(s)+\bar q(s)}{2}\Bigr)^n ,
\end{equation}
and in particular $\norm{\mathcal F(P,Q)(\cdot,s)}_\infty\to0$ as $s\to\infty$.

\emph{Step 1: $Q$ on the slab.} Let $t\ge\rho$ and $\abs{x+t}\le\rho$. Then
$x-t\le\rho-2t\le-\rho$, so that $Q_0(x-t)=0$, and Lemma~\ref{lem:window} gives
$$
0\le Q(x,t)\le\rho\sup_{t-\rho\le s\le t}\norm{\mathcal F(P,Q)(\cdot,s)}_\infty .
$$
Let $\epsilon(s):=\sup_{s'\ge s}\bigl(p(s')+\bar q(s')\bigr)$, which tends to $0$ as
$s\to\infty$ by the last two displays. Since $p$ is nonincreasing, \eqref{eq:F-slab} gives
\begin{equation}\label{eq:qbar}
0\le\bar q(t)\le\rho\,2^{-n}p(t-\rho)^m\epsilon(t-\rho)^n ,
\qquad t\ge\rho .
\end{equation}
Next, let $\xi_1$ be a point with $\abs{Y_{\xi_1}(t-\rho)}=p(t-\rho)$. By
\eqref{eq:Y-ODE} and \eqref{eq:F-slab},
$$
p(t)\ge\abs{Y_{\xi_1}(t)}\ge p(t-\rho)-\rho\sup_{t-\rho\le s\le t}
\norm{\mathcal F(P,Q)(\cdot,s)}_\infty
\ge p(t-\rho)\bigl(1-\rho\,2^{-n}p(t-\rho)^{m-1}\epsilon(t-\rho)^n\bigr),
$$
and the last factor is at least $\frac12$ for large $t$, because $m\ge1$, $p$ is bounded
and $\epsilon\to0$. Together with \eqref{eq:qbar}, this yields $t_*\ge\rho$ such that
\begin{equation}\label{eq:qbar-p}
\bar q(t)\le\delta_1(t)\,p(t),
\qquad
\delta_1(t):=\rho\,2^{1-n}p(t-\rho)^{m-1}\epsilon(t-\rho)^n ,
\qquad t\ge t_* ,
\end{equation}
with $\delta_1(t)\to0$ as $t\to\infty$. For $\xi\in[-\rho,\rho]$ and $t\ge t_*$ we write
$y_\xi(t):=\abs{Y_\xi(t)}=-Y_\xi(t)$, and $q_\xi(t)=Q(\xi-t,t)\in[0,\bar q(t)]$, so that,
by \eqref{eq:Y-ODE},
\begin{equation}\label{eq:y-eq}
y_\xi'=-y_\xi^m\Bigl|\frac{q_\xi-y_\xi}{2}\Bigr|^n .
\end{equation}

\emph{Step 2: upper bound.} Let $\delta\in(0,\frac14)$, choose $t_0\ge t_*$ such that
$\delta_1\le\delta$ on $[t_0,\infty)$, put $c:=2^{-n}(1-2\delta)^n$ and let
$$
h(t):=\bigl(p(t_0)^{1-\alpha}+(\alpha-1)c\,(t-t_0)\bigr)^{-1/(\alpha-1)},
$$
the solution of $h'=-ch^{\alpha}$ with $h(t_0)=p(t_0)$. We claim that $p\le h$ on
$[t_0,\infty)$. Otherwise
$t_4:=\sup\{t\ge t_0:\ p\le h\text{ on }[t_0,t]\}$ is finite, $p(t_4)\le h(t_4)$, and
$p<2h$ on some interval $[t_4,t_5]$ with $t_5>t_4$. Fix $\xi\in[-\rho,\rho]$. At every
$t\in[t_4,t_5]$ with $y_\xi(t)>h(t)$ we have
$q_\xi(t)\le\delta p(t)<2\delta h(t)<2\delta y_\xi(t)$, so that
$y_\xi-q_\xi>(1-2\delta)y_\xi$ and, by \eqref{eq:y-eq}, $y_\xi'\le-c\,y_\xi^{\alpha}$. If
$y_\xi(t_2)>h(t_2)$ for some $t_2\in(t_4,t_5]$, let
$t_3:=\sup\{t\in[t_4,t_2]:\ y_\xi(t)\le h(t)\}$; since $y_\xi(t_4)\le p(t_4)\le h(t_4)$, we
have $y_\xi(t_3)=h(t_3)$ and $y_\xi>h$ on $(t_3,t_2]$, where
$(y_\xi-h)'\le-c\,(y_\xi^{\alpha}-h^{\alpha})<0$. This gives $y_\xi(t_2)<h(t_2)$, a
contradiction. Hence $y_\xi\le h$ on $[t_4,t_5]$ for every $\xi$, that is, $p\le h$ on
$[t_0,t_5]$, which contradicts the definition of $t_4$. Consequently
$$
\limsup_{t\to\infty}t^{1/(\alpha-1)}p(t)
\le\bigl((\alpha-1)c\bigr)^{-1/(\alpha-1)}
=\Bigl(\frac{2^{\,n}}{(\alpha-1)(1-2\delta)^n}\Bigr)^{1/(\alpha-1)} .
$$

\emph{Step 3: lower bound.} Let $\theta\in(0,1)$, let $\delta\in(0,\frac14)$ be so small
that $\bigl((1+\theta)/(1-2\delta)\bigr)^{n/(\alpha-1)}\le2$, and choose $t_0\ge t_*$ such
that $\delta_1\le\min\{\delta,\theta/2\}$ on $[t_0,\infty)$. With $h$ as in Step~2, let
$$
g(t):=\bigl(p(t_0)^{1-\alpha}+(\alpha-1)2^{-n}(1+\theta)^n(t-t_0)\bigr)^{-1/(\alpha-1)} .
$$
Comparing the two explicit formulas gives
$h\le\bigl((1+\theta)/(1-2\delta)\bigr)^{n/(\alpha-1)}g\le2g$, so that, by Step~2 and
\eqref{eq:qbar-p}, $\bar q\le\delta_1p\le\delta_1h\le\theta g$ on $[t_0,\infty)$. Let
$\xi_0$ be a point with $y_{\xi_0}(t_0)=p(t_0)$, and write $y:=y_{\xi_0}$ and
$q:=q_{\xi_0}$. We claim that $y\ge g$ on $[t_0,\infty)$. Otherwise, since
$y(t_0)=g(t_0)$, there are $t_3<t_2$ with $y(t_3)=g(t_3)$ and $y<g$ on $(t_3,t_2]$. There
$\abs{q-y}\le q+y\le(1+\theta)g$, so that \eqref{eq:y-eq} gives
$y'\ge-2^{-n}(1+\theta)^ny^mg^n$, whereas $g'=-2^{-n}(1+\theta)^ng^mg^n$. Hence
$(y-g)'\ge2^{-n}(1+\theta)^ng^n(g^m-y^m)>0$ on $(t_3,t_2]$, and $y(t_2)>g(t_2)$, a
contradiction. Therefore $p\ge y\ge g$ on $[t_0,\infty)$, and
$$
\liminf_{t\to\infty}t^{1/(\alpha-1)}p(t)
\ge\Bigl(\frac{2^{\,n}}{(\alpha-1)(1+\theta)^n}\Bigr)^{1/(\alpha-1)} .
$$
Letting $\delta\to0$ in Step~2 and $\theta\to0$ in Step~3 proves \eqref{eq:decay}.

\emph{Step 4: the limit profile.} By Lemma~\ref{lem:localization}, the integrand in \eqref{eq:Qinfty} is continuous,
nonnegative, and vanishes unless $\abs{\zeta+2s}\le\rho$. For $\zeta$ in a compact interval
$[a,b]$ it therefore vanishes for $s>(\rho-a)/2$, so that the integral may be taken over the
fixed interval $[0,(\rho-a)/2]$, on which $(\zeta,s)\mapsto\mathcal F(P,Q)(\zeta+s,s)$ is
uniformly continuous for $\zeta\in[a,b]$; hence $Q_\infty$ is finite and continuous. It vanishes on
$[\rho,\infty)$, where $Q_0=0$ and $\zeta+2s>\rho$ for $s>0$. For $\zeta<-\rho$ the
integrand vanishes for $s<(-\rho-\zeta)/2$, so that
$0\le Q_\infty(\zeta)\le\rho\sup_{s\ge(-\rho-\zeta)/2}\norm{\mathcal F(P,Q)(\cdot,s)}_\infty$,
which tends to $0$ as $\zeta\to-\infty$. Now let $t\ge0$, $x\in\R$ and $\zeta:=x-t$. By
\eqref{eq:Q-int},
$$
Q_\infty(\zeta)-Q(x,t)=\int_t^\infty\mathcal F(P,Q)(\zeta+s,s)\,ds .
$$
If $x+t\ge\rho$, then $\zeta+2s>x+t\ge\rho$ for $s>t$, and the right-hand side vanishes.
If $x+t<\rho$, the right-hand side is the integral of a nonnegative function over a set of
length at most $\rho$ contained in $[t,\infty)$, hence at most
$\rho\sup_{s\ge t}\norm{\mathcal F(P,Q)(\cdot,s)}_\infty$. This proves \eqref{eq:Q-limit} and completes the proof.
\end{proof}

The rate \eqref{eq:decay} is the time reversal of the blow-up rate of
Section~\ref{sec:blowup}: along the characteristics on which $P$ is largest, the second
component eventually becomes negligible, and $P$ follows the scalar equation
$y'=2^{-n}\abs{y}^{m+n}$ from negative values. Lemma~\ref{lem:absorption} does not use the
smallness of $\eta$; only the identification of the rate and of the limit profile does,
through \eqref{eq:global-bounds}.

\section{Endpoint exponents and general sources}\label{sec:extensions}

The arguments of Sections~\ref{sec:wellposed}--\ref{sec:global} use the powers only
through three properties of the source: it is nonnegative and bounded by a function of
$\abs{P}+\abs{Q}$, it is locally Lipschitz, and along each characteristic of the first
family it vanishes at $P=0$ at least linearly. We first show that this allows the endpoint
exponents $m,n\ge1$, and then replace the powers by general nonlinear factors.

\subsection{The endpoint exponents}\label{subsec:mn-ge-1}

For $m=1$ or $n=1$ the source $\mathcal F$ is no longer of class $C^1$, and solutions of
\eqref{eq:P-int}--\eqref{eq:Q-int} are in general Lipschitz rather than $C^1$. For $T>0$
we call a pair $(P,Q)$ of bounded continuous functions on $\R\times[0,T]$ which are
Lipschitz continuous in $x$, uniformly in $t\in[0,T]$, and satisfy
\eqref{eq:P-int}--\eqref{eq:Q-int} at every point a \emph{Lipschitz solution} on
$\R\times[0,T]$, and we denote by $T(\eta)$ the supremum of the times $T$ for which such
a solution exists. For a Lipschitz solution, the function $v$ defined by
\eqref{eq:reconstruction} is of class $C^1$ with Lipschitz gradient, satisfies
$v_t=(P+Q)/2$ and $v_x=(P-Q)/2$, and solves \eqref{eq:main} almost everywhere; this is
shown at the end of the proof below. For $m,n>1$, Lipschitz solutions with the data of
\eqref{eq:standing} are of class $C^1$ by Proposition~\ref{prop:wellposed}(b), so that this
lifespan agrees with \eqref{eq:lifespan-def}.

\begin{theorem}\label{thm:mn-ge-1}
Let $m\ge1$ and $n\ge1$, and let $\varphi\in C^2_0(\R)$ and $\psi\in C^1_0(\R)$ satisfy
\eqref{eq:support} with $(\varphi,\psi)\not\equiv(0,0)$. Then Theorem~\ref{thm:main}, with
the same constants, holds for Lipschitz solutions and their lifespan; in particular it
holds for $m=n=1$.
\end{theorem}

\begin{proof}
We indicate the changes. For $m,n\ge1$ and $\abs{P_i}+\abs{Q_i}\le R$, $i=1,2$, write
$V_i=(P_i+Q_i)/2$. Since $\abs{P_i}\le R$ and $\abs{V_i}\le R/2$,
\begin{align*}
\abs{\mathcal F(P_1,Q_1)-\mathcal F(P_2,Q_2)}
&\le\bigl|\abs{P_1}^m-\abs{P_2}^m\bigr|\abs{V_1}^n
+\abs{P_2}^m\bigl|\abs{V_1}^n-\abs{V_2}^n\bigr|\\
&\le2^{-n}(m+n)R^{m+n-1}\bigl(\abs{P_1-P_2}+\abs{Q_1-Q_2}\bigr),
\end{align*}
so that \eqref{eq:F-lip} holds with the same constant $\kappa_1\ge2^{-n}(m+n)$, and the
first inequality in \eqref{eq:F-bounds} is unchanged. These are the only properties of
$\mathcal F$ used in Proposition~\ref{prop:wellposed}(a) and in the proof of the lower bound in Section~\ref{sec:wellposed}.

In the proof of Proposition~\ref{prop:wellposed}(b) in Appendix~\ref{app:local}, the contraction
argument is run, with the same $\tau(K)$ and with data $p,q$ that are merely bounded and
Lipschitz, in the set of pairs $(P,Q)\in\mathcal B$ which are Lipschitz in $x$ with
$\mathrm{Lip}(P)+\mathrm{Lip}(Q)\le2K_1$, where $K_1:=\mathrm{Lip}(p)+\mathrm{Lip}(q)$.
This set is closed under uniform convergence, and it is invariant under the map $\Psi$:
with $L:=\kappa_1(2K)^{m+n-1}$, the sum of the Lipschitz constants of the two components
of $\Psi(P,Q)$ is at most $K_1+2\tau L\cdot2K_1\le2K_1$, because $4\tau(K)L=1$. Hence the
fixed point is a Lipschitz solution. In the proof of Proposition~\ref{prop:wellposed}(c) the restarted data $P(\cdot,t_0)$ and $Q(\cdot,t_0)$ are bounded and Lipschitz, which is all that this version of the existence statement requires, and the concatenated pair is again a Lipschitz solution; the regularity statement in Proposition~\ref{prop:wellposed}(b), where the derivatives of $\mathcal F$ enter, is not used.

The map $t\mapsto Y_\xi(t)=P(\xi-t,t)$ is of class $C^1$ directly from \eqref{eq:P-int},
so that \eqref{eq:Y-ODE} holds, and the proofs in Sections~\ref{sec:blowup} and~\ref{sec:global} are unaffected; in the proof of Theorem~\ref{thm:main}(iii), $m\ge1$ is used to keep $p^{m-1}$ bounded in \eqref{eq:qbar-p}. The function $f(t,y)$ in the proof of
Lemma~\ref{lem:localization} is locally Lipschitz in $y$ for $m,n\ge1$, so that Lemma~\ref{lem:localization} and \eqref{eq:signs} hold, the latter including uniqueness backward in time. The remaining arguments use only the nonnegativity of the source and
the first inequality in \eqref{eq:F-bounds}.

Finally, let $(P,Q)$ be a Lipschitz solution on $\R\times[0,T]$. By
\eqref{eq:P-int}--\eqref{eq:Q-int}, the Lipschitz continuity of $P_0$, $Q_0$ and of
$(P,Q)$ in $x$, and \eqref{eq:F-lip}, the functions $P$ and $Q$ are also Lipschitz in $t$,
hence differentiable almost everywhere. Since $s\mapsto P(\xi-s,s)$ and
$s\mapsto Q(\zeta+s,s)$ are of class $C^1$ with derivative $\mathcal F(P,Q)$, the
equations \eqref{eq:system-again} hold at every point of differentiability, and
subtracting them gives $P_x+Q_x=P_t-Q_t$ almost everywhere. For $v$ given by
\eqref{eq:reconstruction} one has $v_t=(P+Q)/2$, and for every $t$ and every
$\chi\in C_0^\infty(\R)$, integration by parts in $x$ and Fubini's theorem give
\begin{align*}
\int_\R v(x,t)\chi'(x)\,dx
&=-\int_\R\Bigl(\eta\varphi'(x)+\frac12\int_0^t(P_t-Q_t)(x,s)\,ds\Bigr)\chi(x)\,dx\\
&=-\int_\R\frac{P(x,t)-Q(x,t)}{2}\,\chi(x)\,dx ,
\end{align*}
because $P(x,0)-Q(x,0)=2\eta\varphi'(x)$. Hence $v_x=(P-Q)/2$, which is continuous, so that
$v\in C^1$ with Lipschitz gradient, and
$v_{tt}-v_{xx}=\frac12\bigl[(P_t-P_x)+(Q_t+Q_x)\bigr]=\mathcal F(P,Q)$ almost everywhere.
\end{proof}

The regularity of the data enters this proof only through the Lipschitz continuity of
$P_0$ and $Q_0$, so Theorem~\ref{thm:mn-ge-1} holds whenever $\varphi'$ and $\psi$ are
Lipschitz continuous with support in $[-\rho,\rho]$.

The threshold $m\ge1$ is required by the localization itself. For $0<m<1$ the scalar
equation \eqref{eq:Y-ODE} with $Y_\xi(0)=0$ may have solutions that leave zero, the zero
set of $P$ is no longer invariant, and the arguments of Sections~\ref{sec:blowup} and~\ref{sec:global} that rely on Lemma~\ref{lem:window} no longer apply. The restriction $n\ge1$ plays a different role.
It guarantees the local Lipschitz continuity of the source across $P+Q=0$, which is what
the well-posedness and continuation theory of Section~\ref{sec:wellposed} uses; the
arguments along characteristics need only $n>0$.

\subsection{General sources}\label{subsec:general}

We now consider
\begin{equation}\label{eq:general}
v_{tt}-v_{xx}=f(v_t+v_x)\,g(v_t),
\qquad
v(x,0)=\eta\varphi(x),\quad v_t(x,0)=\eta\psi(x),
\end{equation}
with $\varphi,\psi$ as in Theorem~\ref{thm:mn-ge-1}, and put
\begin{equation}\label{eq:Phi-Theta}
\Phi(s):=f(s)\,g\Bigl(\frac s2\Bigr),
\qquad
\Theta(y):=\int_y^\infty\frac{ds}{\Phi(s)}\in(0,\infty],
\qquad s,y>0 .
\end{equation}
The nonlinear factors are assumed to satisfy the following hypotheses.
\begin{enumerate}[label=\textup{(H\arabic*)},leftmargin=3em]
\item\label{H1} $f,g:\R\to[0,\infty)$ are continuous, even and nondecreasing on
$[0,\infty)$, with $f(0)=0$, and $f(s)>0$, $g(s)>0$ for $s\ne0$.
\item\label{H2} $f$ and $g$ are locally Lipschitz on $\R$.
\item\label{H3} There are $\alpha>1$ and $C_\Phi\ge1$ such that $s\mapsto\Phi(s)/s^\alpha$
is nondecreasing on $(0,\infty)$ and $\Phi(2s)\le C_\Phi\,\Phi(s)$ for $0<s\le1$.
\end{enumerate}
By \ref{H1}, $\Phi$ is positive and nondecreasing on $(0,\infty)$, and $\Theta$ is an
integral of Osgood type \cite{Osgood1898}: the solution of $y'=\Phi(y)$ issued from $y>0$
blows up in finite time if and only if $\Theta(y)<\infty$, and its blow-up time is then
$\Theta(y)$. For $f(s)=\abs{s}^m$ and $g(s)=\abs{s}^n$ with $m,n\ge1$ one has
$\Phi(s)=2^{-n}s^{m+n}$, hypotheses \ref{H1}--\ref{H3} hold with $\alpha=m+n$ and
$C_\Phi=2^{m+n}$, and $\Theta(\eta P_0^{\max})=C_1\eta^{-(m+n-1)}$.

In the characteristic variables \eqref{eq:PQ-def}, the source of \eqref{eq:general} is
$f(P)g\bigl((P+Q)/2\bigr)$, and \ref{H1} gives, for $\abs{P}+\abs{Q}\le R$,
\begin{equation}\label{eq:Lambda}
0\le f(P)\,g\Bigl(\frac{P+Q}{2}\Bigr)\le\Phi(R),
\end{equation}
since $\abs{P}\le R$ and $\abs{P+Q}/2\le R/2$. Solutions are understood as Lipschitz solutions in the sense of Section~\ref{subsec:mn-ge-1}, with $\mathcal F$ replaced by this source, and $T(\eta)$
denotes their lifespan.

\begin{theorem}\label{thm:general}
Assume \ref{H1} and \ref{H2}.
\begin{enumerate}[label=\textup{(\roman*)},leftmargin=2.4em]
\item For every $\eta>0$,
$\displaystyle T(\eta)\ge\int_{\eta E_0}^{2\eta E_0}\frac{ds}{2\Phi(s)}$.
\item If $P_0^{\max}>0$ and $\Theta(\eta P_0^{\max})<\infty$, then
$T(\eta)\le\rho+\Theta(\eta P_0^{\max})$.
\item If $P_0\le0$ on $\R$ and some $M>0$ satisfies
$\eta\norm{Q_0}_\infty+\rho\,\Phi\bigl(\eta\norm{P_0}_\infty+M\bigr)<M$, then
$T(\eta)=+\infty$.
\item Assume in addition \ref{H3}. Then $\Theta(y)<\infty$ for every $y>0$, the condition
in \textup{(iii)} holds for all sufficiently small $\eta>0$, and if $P_0^{\max}>0$,
\begin{equation}\label{eq:general-sharp}
\lim_{\eta\to0^+}\frac{T(\eta)}{\Theta(\eta P_0^{\max})}=1,
\qquad
T(\eta)\asymp\frac{\eta}{\Phi(\eta)}\quad\text{as }\eta\to0^+ .
\end{equation}
\end{enumerate}
\end{theorem}

\begin{proof}
By \ref{H2}, the source $f(P)g((P+Q)/2)$ is locally Lipschitz, with some Lipschitz
constant $L(R)$ on $\{\abs{P}+\abs{Q}\le R\}$, and by \eqref{eq:Lambda} it is bounded there
by $\Phi(R)$. The proof of Theorem~\ref{thm:mn-ge-1} therefore applies with
$\kappa_0R^{m+n}$ and $\kappa_1R^{m+n-1}$ replaced by $\Phi(R)$ and $L(R)$: Lipschitz
solutions are unique and propagate with speed one, they exist locally on a time interval
depending only on the supremum norm of the data, and the continuation criterion of Proposition~\ref{prop:wellposed}(c) holds. Along a characteristic of the first family,
$Y_\xi$ solves $y'=f(y)g\bigl((y+q_\xi(t))/2\bigr)$, whose right-hand side is locally
Lipschitz in $y$ and vanishes at $y=0$. Hence \eqref{eq:Y-ODE}, \eqref{eq:monotone}, Lemma~\ref{lem:localization}, \eqref{eq:signs} and the waiting-time property \eqref{eq:Q-nonneg-late} hold, and so does Lemma~\ref{lem:window}, with $\kappa_0(\cdot)^{m+n}$ replaced by $\Phi(\cdot)$ in \eqref{eq:window}.

(i) As in Section~\ref{sec:wellposed},
$K(t):=\norm{P(\cdot,t)}_\infty+\norm{Q(\cdot,t)}_\infty$ satisfies
$K\le W:=\eta E_0+2\int_0^t\Phi(K(s))\,ds$ by \eqref{eq:Lambda}. Since
$W'=2\Phi(K)\le2\Phi(W)$ and $\Phi>0$ on $(0,\infty)$, we get
$\int_{\eta E_0}^{W(t)}ds/(2\Phi(s))\le t$, so that $K(t)\le W(t)<2\eta E_0$ for
$t<\int_{\eta E_0}^{2\eta E_0}ds/(2\Phi(s))$, and the continuation criterion gives (i).

(ii) Let $x_0$ be a maximum point of $P_0$ and $t_1=(x_0+\rho)/2$. For $t\ge t_1$ we have
$Y:=Y_{x_0}\ge\eta P_0^{\max}>0$ and $q:=q_{x_0}\ge0$, so that $(Y+q)/2\ge Y/2>0$ and, by
\ref{H1}, $Y'\ge f(Y)g(Y/2)=\Phi(Y)$. Hence
$t-t_1\le\int_{Y(t_1)}^{Y(t)}ds/\Phi(s)<\Theta(\eta P_0^{\max})$ for all
$t\in[t_1,T(\eta))$, which gives (ii) since $t_1\le\rho$.

(iii) By \eqref{eq:signs}, $\norm{P(\cdot,t)}_\infty\le\eta\norm{P_0}_\infty$. As in the
proof of Theorem~\ref{thm:main}(ii), if $\norm{Q(\cdot,t_2)}_\infty=M$ at a first time
$t_2<T(\eta)$, then \eqref{eq:window}, with $\kappa_0(\cdot)^{m+n}$ replaced by
$\Phi(\cdot)$, would give $M\le\eta\norm{Q_0}_\infty+\rho\,\Phi(\eta\norm{P_0}_\infty+M)<M$.
Hence $\norm{Q(\cdot,t)}_\infty<M$ for $t<T(\eta)$, and the continuation criterion gives
$T(\eta)=+\infty$.

(iv) Assume \ref{H3}. For $s\ge y>0$ it gives $\Phi(s)\ge\Phi(y)(s/y)^\alpha$, hence
\begin{equation}\label{eq:Theta-upper}
\Theta(y)\le\frac{y}{(\alpha-1)\Phi(y)}<\infty .
\end{equation}
It also gives $\Phi(s)\le\Phi(1)s^\alpha$ for $0<s\le1$, so that $\Phi(s)/s\to0$ as
$s\to0^+$. With $B:=\norm{Q_0}_\infty$ and $M:=\eta(B+1)$, the condition in (iii) then
holds as soon as $\rho\,\Phi\bigl(\eta(\norm{P_0}_\infty+B+1)\bigr)<\eta$, which is the case
for all small $\eta$.

Next, for fixed $\lambda>0$, monotonicity and the doubling property in \ref{H3} give
$\Phi(\lambda\eta)\asymp\Phi(\eta)$ for small $\eta>0$, with constants depending on
$\lambda$. Moreover $\eta/\Phi(\eta)\to\infty$ as $\eta\to0^+$. By (ii) and
\eqref{eq:Theta-upper},
$T(\eta)\le\rho+\eta P_0^{\max}/\bigl((\alpha-1)\Phi(\eta P_0^{\max})\bigr)\le C\eta/\Phi(\eta)$
for small $\eta$, and by (i),
$T(\eta)\ge\eta E_0/\bigl(2\Phi(2\eta E_0)\bigr)\ge c\,\eta/\Phi(\eta)$. This proves the
second assertion in \eqref{eq:general-sharp}.

It remains to prove the limit in \eqref{eq:general-sharp}. Write $\beta:=P_0^{\max}$. For
$\lambda\ge1$ and $0<y\le\frac12$, \ref{H3} gives
\begin{equation}\label{eq:Theta-scaling}
\Theta(\lambda y)\le\lambda^{1-\alpha}\,\Theta(y),
\qquad
\Theta(\lambda y)\ge\bigl(1-(\lambda-1)C_\Phi\bigr)\Theta(y).
\end{equation}
The first inequality follows from $\Theta(\lambda y)=\lambda\int_y^\infty\Phi(\lambda
u)^{-1}\,du$ and $\Phi(\lambda u)\ge\lambda^\alpha\Phi(u)$. For the second,
$\Theta(y)-\Theta(\lambda y)\le(\lambda-1)y/\Phi(y)$ because $\Phi$ is nondecreasing, while
$\Theta(y)\ge\int_y^{2y}\Phi(s)^{-1}\,ds\ge y/\Phi(2y)\ge y/\bigl(C_\Phi\Phi(y)\bigr)$.
Since $\Theta(\eta\beta)\to\infty$ as $\eta\to0^+$, (ii) gives
$\limsup_{\eta\to0^+}T(\eta)/\Theta(\eta\beta)\le1$. For the lower bound, fix
$\varepsilon\in(0,1)$, then $M\ge2\beta$ with $(M/\beta)^{1-\alpha}\le\varepsilon/3$ and
$\delta\in(0,\beta]$ with $\delta C_\Phi/\beta\le\varepsilon/3$, and put
$M':=\max\{M,\norm{P_0}_\infty\}$ and $R:=M'+B+2$. Since $\Phi(s)/s\to0$ as $s\to0^+$,
there is $\eta_\varepsilon>0$ such that $\eta_\varepsilon\beta\le\frac12$ and
$\rho\,\Phi(\eta R)\le\eta\min\{\frac12,\frac\delta2\}$ for $0<\eta\le\eta_\varepsilon$.
We run the bootstrap in the proof of \eqref{eq:limit} up to the time
$(1-\varepsilon)\Theta(\eta\beta)$, with $2^{-n}(\eta R)^{\alpha}$ replaced by
$\Phi(\eta R)$, which bounds the source by \eqref{eq:Lambda}. As there, we obtain
$\norm{Q(\cdot,t)}_\infty\le\eta(B+\frac12)$, $Y_\xi\le\eta(\beta+\frac\delta2)\le\eta M$
during the waiting time $[0,t_\xi]$, and $0\le q_\xi\le\frac{\eta\delta}2$ after it. The
function $\tilde Y:=Y_\xi+\eta\delta/2$ then satisfies
$\tilde Y'\le f(\tilde Y)g(\tilde Y/2)=\Phi(\tilde Y)$ by \ref{H1}, and
$\tilde Y(t_\xi)\le\eta(\beta+\delta)$. If $\tilde Y$ reached the value $\eta M$ at a time
$t$ of the bootstrap interval, then, by \eqref{eq:Theta-scaling} with $y=\eta\beta$,
$$
(1-\varepsilon)\Theta(\eta\beta)>t-t_\xi
\ge\int_{\eta(\beta+\delta)}^{\eta M}\frac{ds}{\Phi(s)}
=\Theta\bigl(\eta(\beta+\delta)\bigr)-\Theta(\eta M)
\ge\Bigl(1-\frac{2\varepsilon}{3}\Bigr)\Theta(\eta\beta),
$$
which is absurd. Hence $Y_\xi<\eta M$, the bootstrap closes as there, and $T(\eta)\ge(1-\varepsilon)\Theta(\eta\beta)$ for
$0<\eta\le\eta_\varepsilon$.
\end{proof}

The absorption statement of Lemma~\ref{lem:absorption} also holds under
\ref{H1}--\ref{H2}, with the same proof and with $\abs{L}^m(\abs{L}/4)^n$ replaced by
$f(L)g(L/4)>0$; the rate of absorption then depends on the behaviour of $\Phi$ near $0$.
The theorem applies in particular to powers modified by slowly varying factors.

Let $\mu_1,\mu_2:[0,\infty)\to[0,\infty)$ be nondecreasing and positive on $(0,\infty)$,
and assume that $s\mapsto s\mu_i(s)$ is locally Lipschitz on $[0,\infty)$ and that
$\mu_i(2s)\le C\mu_i(s)$ for $0<s\le1$, $i=1,2$. Then $f(s):=\abs{s}\mu_1(\abs{s})$ and
$g(s):=\abs{s}\mu_2(\abs{s})$ satisfy \ref{H1} and \ref{H2}, and
$\Phi(s)=\frac{s^2}{2}\mu_1(s)\mu_2(s/2)$, so that $\Phi(s)/s^2$ is nondecreasing and
$\Phi(2s)=2s^2\mu_1(2s)\mu_2(s)\le4C^2\Phi(s)$ for $0<s\le1$. Hence \ref{H3} holds with
$\alpha=2$, and if $P_0^{\max}>0$, \eqref{eq:general-sharp} gives
$$
T(\eta)\asymp\frac{\eta}{\Phi(\eta)}=\frac{2}{\eta\,\mu_1(\eta)\,\mu_2(\eta/2)}
\qquad\text{as }\eta\to0^+ .
$$

For instance, $\mu_1=\mu_2=\mu$ with $\mu(s):=1/\log(e+1/s)$ for $s>0$ and $\mu(0):=0$
satisfies these hypotheses, with $C=1+\log2$, and gives
$T(\eta)\asymp(\log(1/\eta))^2/\eta$. The logarithmic factors thus lengthen the lifespan
of the case $m=n=1$, for which $T(\eta)\asymp1/\eta$ by Theorem~\ref{thm:mn-ge-1}.

\section{Concluding remarks}\label{sec:conclusion}

The alternative proved above is governed by one initial characteristic component and is
insensitive to the other. The reason lies in the sign of the source rather than in its
size: since $\mathcal F\ge0$, both components increase along their own characteristics, so
that a positive $P$ is amplified, while a negative $P$ is driven towards $0$ and, by \eqref{eq:signs}, cannot cross it. For an odd source such as
$\abs{v_t\pm v_x}^{p-1}(v_t\pm v_x)$ the substitution $v\mapsto-v$ restores the symmetry
between the two signs, and the blow-up condition can then be stated without a sign condition; see
\cite{HaruyamaSasakiTakamura2025}. Here that substitution changes the equation, and no
symmetry between the two signs is available.

The form of the second factor matters at one point only. The upper bound in Theorem~\ref{thm:main}(i) uses that $\mathcal F(P,Q)\ge\mathcal F(P,0)$ for $P>0$ and
$Q\ge0$, a monotonicity that belongs to the velocity $v_t=(P+Q)/2$. For the source
$\abs{v_t+v_x}^m\abs{v_x}^n$, whose characteristic form is
$\abs{P}^m\abs{(P-Q)/2}^n$, the localization of Section~\ref{sec:structure}, the lower bound in \eqref{eq:two-sided} and part (ii) of Theorem~\ref{thm:main} hold with the same proofs; but along the selected characteristic the factor $\abs{P-Q}$ may become small when $Q\ge0$, and the proof of the upper bound in Section~\ref{sec:blowup} no longer applies.
Whether the sign of $\psi+\varphi'$ still decides the small-data alternative for that
source is open.

The compact support of the data is used twice, in Lemma~\ref{lem:window} and through the waiting time in \eqref{eq:Q-nonneg-late}. For data that are merely bounded, the second
component may be forced along its characteristics for all times, the bootstrap for $Q$ in the proof of Theorem~\ref{thm:main}(ii) is no longer available, and part (ii) has no counterpart. The blow-up statement, on the other hand, survives under the
additional hypothesis $Q_0\ge0$, which provides the sign \eqref{eq:Q-nonneg-late} directly
and with $t_1=0$, as in Section~\ref{sec:blowup}.

The difference between \eqref{eq:main} and the product source $\abs{v_t}^p\abs{v_x}^q$,
for which the matching upper bound remains open \cite{HaruyamaSasakiTakamura2025}, lies in
the zero sets of the factors. There neither factor is attached to a single characteristic
component: $v_t$ vanishes where $P=-Q$ and $v_x$ where $P=Q$, and neither of these sets is
invariant along a family of characteristics. The localization of
Lemma~\ref{lem:localization} has no analogue, and along the characteristic selected in Section~\ref{sec:blowup} the factor $\abs{v_x}$ may vanish. The present model is the case
in which one factor is attached to the family along which its own component is
transported.

Several questions remain open. By \eqref{eq:upper} and \eqref{eq:limit}, the remainder $T(\eta)-C_1\eta^{-(m+n-1)}$ is of smaller order than $\eta^{-(m+n-1)}$ and is
bounded above by $\rho$; whether it is also bounded below, and how it depends on the
profiles, is not known. Under \eqref{eq:negativity}, the behaviour for large amplitudes is undecided,
since we do not know whether the smallness condition $\eta<\eta_0$ can be removed. It
would also be of interest to know whether the singularity in Theorem~\ref{thm:main}(i) is of gradient type, and whether the blow-up curve can be described as in \cite{Sasaki2018}; the
confinement of the source to the slab $\{\abs{x+t}\le\rho\}$ given by \eqref{eq:slab} is a
natural starting point for such an analysis, although it says nothing by itself about the
location of the singularity. Finally, the analogue of Theorem~\ref{thm:main} for the
general product source remains to be found.

\appendix

\section{Local existence}\label{app:local}

We prove Proposition~\ref{prop:wellposed}(b). Throughout, $t_0\ge0$ and $p,q\in C^1(\R)$ are bounded with bounded and uniformly
continuous derivatives; we fix $K>0$ with $\norm{p}_\infty+\norm{q}_\infty\le K$, and set
$K_1:=\norm{p'}_\infty+\norm{q'}_\infty$ and $\tau:=\tau(K)$.

\begin{proof}[Proof of Proposition~\ref{prop:wellposed}\textup{(b)}]
\emph{Existence.} Let $\mathcal B$ be the set of pairs $(P,Q)$ of continuous functions on
$\R\times[t_0,t_0+\tau]$ with $\norm{(P,Q)}:=\sup\bigl(\abs{P}+\abs{Q}\bigr)\le2K$, which
is a complete metric space for the distance induced by $\norm{\cdot}$. For
$(P,Q)\in\mathcal B$ let $\Psi(P,Q)$ be the pair given by the right-hand sides of \eqref{eq:shifted}; it is continuous, since $\mathcal F(P,Q)$ is
bounded and continuous. By \eqref{eq:F-bounds},
$$
\norm{\Psi(P,Q)}\le K+2\tau\kappa_0(2K)^{m+n}\le2K ,
$$
because $2\tau\kappa_0(2K)^{m+n}=(\kappa_0/\kappa_1)K\le K$. By \eqref{eq:F-lip} with
$R=2K$ and $L:=\kappa_1(2K)^{m+n-1}$,
$$
\norm{\Psi(P_1,Q_1)-\Psi(P_2,Q_2)}
\le2\tau L\,\norm{(P_1-P_2,Q_1-Q_2)}
=\frac12\,\norm{(P_1-P_2,Q_1-Q_2)} ,
$$
since $2\tau(K)L=\frac12$. The contraction mapping theorem gives a fixed point in $\mathcal B$, which is a bounded
continuous solution with $\norm{P}_\infty+\norm{Q}_\infty\le2K$, and by part (a) it is
the only continuous solution.

\emph{Regularity.} Let now $T'>t_0$ and let $(P,Q)$ be any bounded continuous solution of
\eqref{eq:shifted} on $\R\times[t_0,T']$. We show that $P$ and $Q$ are of class $C^1$,
with bounded derivatives. Let $\bar K:=\sup\bigl(\abs{P}+\abs{Q}\bigr)$ and
$L:=\kappa_1\bar K^{m+n-1}$.

\emph{Step 1: Lipschitz bound.} For $h\in\R$ and $t\in[t_0,T']$ let
$$
\omega(t,h):=\sup_{x\in\R}\bigl(\abs{P(x+h,t)-P(x,t)}+\abs{Q(x+h,t)-Q(x,t)}\bigr).
$$
For fixed $h$ this is a bounded function of $t$, lower semicontinuous as a supremum of
continuous functions, hence measurable; continuity is not needed below. Subtracting the
integral equations at $x+h$ and at $x$ and using \eqref{eq:F-lip}, we obtain
$$
\omega(t,h)\le K_1\abs{h}+2L\int_{t_0}^t\omega(s,h)\,ds ,
$$
whence $\omega(t,h)\le K_1\abs{h}\,e^{2L(t-t_0)}$ by Gronwall's inequality.

\emph{Step 2: differentiability in $x$.} The coefficients $a:=\mathcal F_P(P,Q)$ and
$b:=\mathcal F_Q(P,Q)$ are bounded and continuous, so the linear system
\begin{align*}
A(x,t)&=p'(x+t-t_0)+\int_{t_0}^t(aA+bB)(x+t-s,s)\,ds,\\
B(x,t)&=q'(x-t+t_0)+\int_{t_0}^t(aA+bB)(x-t+s,s)\,ds,
\end{align*}
has a unique bounded continuous solution, obtained by the contraction mapping theorem on
successive time intervals of a fixed short length. For $h\ne0$ let
$D_hP(x,t):=\bigl(P(x+h,t)-P(x,t)\bigr)/h$, and define $D_hQ$, $D_hp$ and $D_hq$ in the
same way. By the mean value theorem,
$$
\frac{\mathcal F(P,Q)(x+h,t)-\mathcal F(P,Q)(x,t)}{h}
=\bar a_h\,D_hP+\bar b_h\,D_hQ ,
$$
where $\bar a_h$ and $\bar b_h$ are the averages of $\mathcal F_P$ and $\mathcal F_Q$ over
the segment joining $(P,Q)(x,t)$ to $(P,Q)(x+h,t)$. By Step~1 this segment has length at
most $K_1\abs{h}e^{2L(T'-t_0)}$, and $\nabla\mathcal F$ is uniformly continuous on
$\{\abs{P}+\abs{Q}\le\bar K\}$, so that $\bar a_h-a\to0$ and $\bar b_h-b\to0$ uniformly as
$h\to0$, while $D_hP$ and $D_hQ$ remain bounded. Since $p'$ and $q'$ are uniformly
continuous, $D_hp-p'\to0$ and $D_hq-q'\to0$ uniformly. The function
$\delta_h(t):=\sup_x\bigl(\abs{D_hP-A}+\abs{D_hQ-B}\bigr)(x,t)$, measurable as in Step~1,
therefore satisfies
$$
\delta_h(t)\le\varepsilon_h+2\bigl(\norm{a}_\infty+\norm{b}_\infty\bigr)
\int_{t_0}^t\delta_h(s)\,ds,
\qquad\varepsilon_h\to0 ,
$$
and Gronwall's inequality gives $\delta_h\to0$ uniformly. Hence $P_x=A$ and $Q_x=B$ exist
and are bounded and continuous.

\emph{Step 3: $C^1$ regularity.} In the characteristic variables, $P(x,t)=G(x+t,t)$ with
$$
G(\xi,t):=p(\xi-t_0)+\int_{t_0}^t\mathcal F(P,Q)(\xi-s,s)\,ds .
$$
The partial derivatives $\partial_tG(\xi,t)=\mathcal F(P,Q)(\xi-t,t)$ and
$\partial_\xi G(\xi,t)=P_x(\xi-t,t)$ are continuous, so $G$, and hence $P$, is of class
$C^1$, and $P_t=P_x+\mathcal F(P,Q)$ is bounded. The same argument applies to $Q$.
\end{proof}

\section*{Statements and declarations}

\noindent\textbf{Funding.} No funding was received for this work.

\noindent\textbf{Competing interests.} The author declares no competing interests.

\noindent\textbf{Data availability.} No datasets were generated or analysed in this
study.

\end{document}